\documentclass{amsart}

\usepackage{amsmath,amsfonts,amssymb,amsthm}
\usepackage{hyperref}
\usepackage{mathtools}
\usepackage{pdfpages}
 
\usepackage[capitalize]{cleveref}
\usepackage{stmaryrd}
\usepackage{tikz-cd}

\newtheorem{lem}{Lemma}[section]
\newtheorem{teo}[lem]{Theorem}

\newtheorem{pro}[lem]{Proposition}
\newtheorem{cor}[lem]{Corollary}
\newtheorem{claim}[lem]{Claim}

\newtheorem*{con*}{Conjecture}

\newtheorem{Conj}{Conjecture}
\newtheorem{Question}[Conj]{Question}

\theoremstyle{definition}
\newtheorem{exa}[lem]{Example}

\theoremstyle{remark}
\newtheorem*{rem*}{Remark}

\newcommand{\argu}{\hbox to 7truept{\hrulefill}}

\DeclareMathOperator{\D}{\mathcal D}

\DeclareMathOperator{\SL}{SL}

\DeclareMathOperator{\im}{Im}

\DeclareMathOperator{\rk}{rk}

\DeclareMathOperator{\Ann}{Ann}

\DeclareMathOperator{\supp}{supp}

\DeclareMathOperator{\ab}{ab}

\DeclareMathOperator{\Tor}{Tor}

\newcommand{\myeq}[1]{\ensuremath{\stackrel{\text{#1}}{=}}}

\newcommand{\R}{\mathbb{R}}

\newcommand{\Z}{\mathbb{Z}}

\newcommand{\N}{\mathbb{N}}
\newcommand{\CC}{\mathbb{C}}
\newcommand{\Q}{\mathbb{Q}}

\DeclareMathOperator{\fp}{FP}
\def\immerses{\looparrowright}
\newcommand{\isom}{\cong}
\newcommand{\normal}[1]{\langle\!\langle #1 \rangle\!\rangle}

 \date{\today}
 
\begin{document}

\title[Coherence of one-relator groups]{On the coherence of  one-relator groups and their group algebras}
\author{Andrei Jaikin-Zapirain}

 \address{Departamento de Matem\'aticas, Universidad Aut\'onoma de Madrid \and  Instituto de Ciencias Matem\'aticas, CSIC-UAM-UC3M-UCM, Madrid, Spain}
\email{andrei.jaikin@icmat.es}

\author{Marco Linton}
\address{Instituto de Ciencias Matem\'aticas, CSIC-UAM-UC3M-UCM, Madrid, Spain}
\email{marco.linton@icmat.es}

\begin{abstract}
We prove that one-relator groups are coherent, solving a well-known problem of Gilbert Baumslag. Our proof strategy is readily applicable to many classes of groups of cohomological dimension two. We show that fundamental groups of two complexes with non-positive immersions are homologically coherent, we show that groups with staggered presentations and many Coxeter groups are coherent and we show that group algebras over fields of characteristic zero of groups with reducible presentations without proper powers are coherent.
\end{abstract}

\maketitle

\section{Introduction}
A group is {\bf coherent} if all of its finitely generated subgroups are finitely presented. A notorious conjecture of Baumslag's \cite{Ba74} predicts that all one-relator groups are coherent. The first non-trivial result on the coherence of one-relator groups is due to Karrass--Solitar \cite{KS70,KS71} who showed that all cyclically and conjugacy pinched one-relator groups are coherent. Renewed interest in Baumslag's conjecture in the last two decades has galvanised much of the recent literature on coherence, as can be seen from Wise's excellent survey \cite{Wi20}.

A notion that plays a particularly important role in recent developments is that of {\bf non-positive immersions}. Wise introduced this notion in \cite{wise_03} and laid out a program to solve Baumslag's conjecture. The strategy involved showing that the fundamental group of a two-complex with non-positive immersions is coherent and establishing that the presentation complex of a torsion-free one-relator group has non-positive immersions. Although the former remains open, many groups known (or conjectured) to be coherent have been shown to fall under the rubric of non-positive immersion \cite{Wi20}. The latter step was eventually achieved by Helfer--Wise \cite{HW16} and, independently, Louder--Wilton \cite{LW17}.

Following on from this progress, Louder--Wilton \cite{louder_20} and, independently, Wise \cite{wise_20} showed that all one-relator groups with torsion are coherent. Both proofs relied on the fact that a one-relator group with torsion has a finite index subgroup that is the fundamental group of a two-complex with a stronger version of non-positive immersions. Another, related, strengthening of non-positive immersions is constituted by negative immersions, as defined in \cite{LW22}. Louder--Wilton showed in \cite{LW22} that the presentation complex of a one-relator group $G$ has negative immersions if and only if $G$ is two-free and then later confirmed Baumslag's conjecture for all such one-relator groups in \cite{LW21}.

In a different direction, one-relator groups were shown to be generically coherent by work of Sapir--\v{S}pakulov\'{a} \cite{SS10} and Kielak--Kropholler--Wilkes \cite{kielak_22}. These results were achieved by showing that generic one-relator groups are virtually ascending HNN-extensions of free groups and then appealing to a result of Feighn--Handel \cite{FH99}. In a similar vein, Kielak and the second author showed in \cite{KL23} that hyperbolic and virtually compact special one-relator groups are virtually free-by-cyclic and then also appealed to the work of Feighn--Handel to establish their coherence.

Although significant work has gone into Baumslag's conjecture, the authors are not aware of any literature on the coherence of group rings of one-relator groups; a ring $R$ is  (left)  {\bf coherent} if every  finitely generated (left) ideal is finitely presented as a (left) $R$-module. There seems to be a relation between the coherence of $G$ and $\Z[G]$. They are equivalent for elementary amenable  groups \cite{HKKL23}  (see also \cite{BS79})  and they are both commensurability invariants. However, in general there is no known implication between the two properties.

In this paper we confirm Baumslag's conjecture for all one-relator groups and their group rings.

\begin{teo}
\label{main}
If $G$ is a one-relator group and $K$ is a field of characteristic zero, then 
\begin{enumerate}
\item $G$ is coherent.
\item $K[G]$ is coherent.
\end{enumerate}
\end{teo}

We now explain the steps that go into the proof.

\subsection{Coherence of groups}

A third notion of coherence plays an important role in our proof of \cref{main}; a group $G$ is {\bf homologically coherent} if every finitely generated subgroup of $G$ is of type $\fp_2\left (\Z\right )$. Both coherence of $G$ and $\Z[G]$ imply homological coherence of $G$. The first groups of type $\fp_2(\Z)$ that are not finitely presented were constructed by Bestvina--Brady \cite{bestvina_97}. Building on their ideas, many groups with various finiteness properties have been constructed by a plethora of different authors. Despite this, it is unknown whether there exists a group that is homologically coherent but not coherent. In this direction, Gersten \cite{Ge96} showed that for hyperbolic groups of cohomological dimension two, these two notions of coherence are equivalent.

The proof of the coherence of one-relator groups consists of two main steps. First, we show that one-relator groups are homologically coherent. In fact, we prove homological coherence for the  class of fundamental groups of two-complexes with non-positive immersions and for the class of locally indicable groups of cohomological dimension two with trivial second $L^2$-Betti number. A one-relator group contains a torsion-free subgroup of finite index that belongs to both classes \cite{Br84, Ho00, HW16, LW17}.

\begin{teo}
\label{npi}
Let $G$ belong to one the following families of groups:
\begin{enumerate}
\item Fundamental groups of two-complexes with non-positive immersions.
\item Locally indicable groups of cohomological dimension two with trivial second $L^2$-Betti number.
\end{enumerate}
Then $G $ is homologically coherent.
\end{teo}

We remark that we may replace local indicability in \cref{npi} by satisfying the Strong Atiyah conjecture. However, we show in \cref{Atiyah_locally} that if $G$ satisfies the strong Atiyah conjecture, has cohomological dimension two and $b_2^{(2)}(G) = 0$, then $G$ is actually locally indicable.

\cref{npi} completes a weaker version of Wise's proposed first step for proving Baumslag's conjecture, partially solving \cite[Conjecture 12.11]{Wi20}. It turns out that this is sufficient. The second step in the proof is a kind of promotion of the property $\fp_2\left (\Z\right ) $ to the property of being finitely presented.

\begin{teo} 
\label{equivalence}
Let $k$ be a commutative ring with $1\neq 0$, let $G$ be a group acting on a tree $\mathcal{T}$ with coherent vertex stabilisers. If $H\le  G$ is of type $\fp_2(k) $, then $H$ is finitely presented. In particular, if $G$ is homologically coherent, then $G$ is coherent.
\end{teo}

From \cref{equivalence} it follows that for a large class of groups, coherence is equivalent to homological coherence, see \cref{hierarchy}. Equipped with \cref{npi} and \cref{equivalence}, we settle Baumslag's conjecture by appealing to the classic Magnus--Moldavanskii hierarchy.

The combination of \cref{npi} and \cref{equivalence} provides a powerful tool for proving coherence of groups. Before discussing the proof of the second part of \cref{main}, we mention some further applications.

Our method gives a new proof of the coherence of an ascending HNN-extension of a free group, originally due to Feighn--Handel \cite{FH99}. The homological coherence in this case can be obtained as a consequence of the vanishing of the second $L^2$-Betti number of these groups and the coherence is then concluded from Theorem \ref{equivalence}.

Jankiewicz--Wise \cite[Conjecture 4.7]{jankiewicz_16} put forward a conjectural picture for the coherence of certain Coxeter groups of virtual cohomological dimension two. Following the same strategy, we are also able to confirm one direction of this conjecture; that is, we show that all such Coxeter groups conjectured to be coherent are indeed coherent. See \cref{sec:raags} for the relevant definitions.

\begin{teo}
\label{coxeter}
Let $G$ be a Coxeter group and suppose that ${\chi}(H)\leqslant 0$ for each Coxeter subgroup $H\le  G$ generated by at least three elements. Then $G$ is coherent.
\end{teo}

An often studied generalisation of one-relator groups is the family of groups with staggered presentations. Such groups appear naturally in the study of one-relator groups and one-relator products, see \cite{howie_84} and \cite{lyndon_01}. The reader is referred to \cref{sec:staggered} for the precise definition of such groups. By work of Helfer--Wise \cite{HW16}, a torsion-free group with a staggered presentation is the fundamental group of a two-complex with non-positive immersions. Hence, the exact same strategy for proving \cref{main}(1) also shows that torsion-free groups with staggered presentations are coherent. Groups with staggered presentations where each relator is a proper power are coherent by work of Wise \cite[Theorem 5.7]{wise_20}.  Finitely generated groups from both of these subclasses have finite index subgroups that are fundamental groups of two-complexes with non-positive immersions. However, when some of the relators are proper powers and some are not, this is no longer the case (see \cref{staggered_example}) and so we cannot directly use \cref{npi}. Nevertheless, with a little extra work, we establish the coherence of groups with staggered presentations, solving a conjecture of Wise \cite[Conjecture 14.10]{Wi20}.

\begin{teo} \label{staggered_coherent}
Groups with staggered presentations are coherent.
\end{teo}

\subsection{Coherence of group algebras}

There are not many papers on coherent group rings. For example we could not find any reference for the coherence of $K[G]$, where $G$ is a free-by-cyclic group and $K$ is a field; although this result can be extracted easily from a result of Henneke--L\'{o}pez-\'{A}lvarez \cite{HL22}. In Corollary  \ref{mappingtorus} we prove the coherence of group algebras  of an ascending HNN-extension of a free group.

Other examples of coherent group rings come from general results on coherent rings.   If $k$ is a commutative Noetherian ring, by work of Bondal--Zhdanovskiy \cite{BZ15}, $k[G]$ is coherent if $G$ is the direct product of a  free  group and an abelian group and, by results of \.{A}berg \cite{Ab82} and Lam \cite{La77}, $k[G]$ is coherent if $G$ belongs to the smallest family of groups containing all virtually polycyclic groups and closed under amalgamated free products and HNN-extensions with   virtually polycyclic edge subgroups.   
We would also like to mention a conjecture for graded algebras: every graded algebra with a single defining relation is {\bf graded coherent}, meaning that an arbitrary  finitely generated left graded ideal is finitely presented (see \cite{Pi08} for partial results).

 Our proof of \cref{main}(2) splits into two cases. When $G$ has torsion, $G$ is virtually free-by-cyclic by \cite{KL23} and so $K[G]$ is coherent for any field $K$. When $G$ is torsion-free, the coherence of $K[G]$ follows from our proof of a more general result on fundamental groups of reducible two-complexes. Reducible two-complexes were first introduced by Howie in \cite{Ho82} and can be thought of as generalisations of staggered two-complexes in that they are `staggered on one side'.  

\begin{teo}\label{th:coherentgroupalgebra}
Let $G$ be the  fundamental group of a finite reducible two-complex  without proper powers and $K$ a field of characteristic $0$. Then the group algebra $K[G]$ is coherent. \end{teo}

Howie proved in \cite[Theorem 4.2]{Ho82} that   the  fundamental groups of  reducible two-complexes  without proper powers  are locally indicable. Our proof of Theorem \ref{th:coherentgroupalgebra} uses the existence of the division ring $\D_{K[G]}$  (see Subsection \ref{dkg}).  It is  the  Hughes-free division $K[G]$-ring associated with locally indicable groups whose uniqueness was proved by Hughes in \cite{Hu70}.  When $K$ has characteristic 0, its existence  follows from the solution of the strong Atiyah conjecture for locally indicable  groups \cite{JL20}. If we knew that this division ring existed for an arbitrary $K$, then we would have the same result for such a $K$ as well.

A new property that we have discovered  in the case of  the  fundamental group of  a reducible two-complex  without proper powers   is the following result. This is the key ingredient of our proof of Theorem \ref{th:coherentgroupalgebra}.

\begin{teo}\label{weak}
Let $K$ be a field and let $G$ be the  fundamental group of  a finite reducible two-complex  without proper powers.  Assume that $\D_{K[G]}$ exists. Then as a right $K[G]$-module, $\D_{K[G]}$ is of weak dimension at most $1$.
\end{teo}
This property was known for free, limit and free-by-cyclic groups $G$.  It does not hold for   locally indicable groups in general: there are locally indicable groups  with non-trivial second $L^2$-Betti number (for example, if $G$ is the direct product of two non-abelian free groups). 

\subsection{The rank-1 Hanna Neumann conjecture}
Let $d\left (G\right ) $ denote the minimal number of generators of a group $G$.
The rank-1 Hanna Neumann conjecture  proposed by Wise in \cite{Wi05} and proved  independently by Helfer and Wise \cite{HW16} and Louder and Wilton \cite{LW17}, is the following  statement: if $W$ is a maximal cyclic subgroup of a free group $F$, then  for any subgroup $U$ of $F$,
$$\sum_{x\in W \backslash  F/U}  d \left (xUx^{-1} \cap W\right ) \leqslant 
  \left \{
   \begin{array}{cc} 
   d\left (U\right )  & \textrm{if\ }  U\le \normal{W} \\
   d\left (U\right )-1  & \textrm{if\ } U\not \le \normal{W} \end{array}\right . 
$$
 It is natural to ask for what family of subgroups $W$ of $F$ the same conclusion holds. In Subsection \ref{sec:staggered} we will introduce {\it strictly reducible}\ subgroups of a free group. Their relation with reducible two-complexes without proper powers is  the following: if the presentation complex of $\langle X|R \rangle$ is reducible without proper powers, then there exists a strictly reducible subgroup $W$ of the free group $F(X)$ such that $\langle X|R \rangle\cong F(X)/\normal{W}$. Using our approach we prove the following generalisation of the rank-1 Hanna Neumann conjecture.

\begin{teo}\label{HN}
Let $F$ be a free group, $U$ a subgroup of $F$ and $W$ a strictly reducible subgroup of $F$. Then
$$\sum_{x\in W \backslash  F/U}  d \left (xUx^{-1} \cap W\right ) \leqslant 
  \left \{
   \begin{array}{cc} 
   d\left (U\right )  & \textrm{if\ }  U\le \normal{W} \\
 d\left (U\right )-1  & \textrm{if\ } U\not \le \normal{W} \end{array}\right . 
$$
\end{teo}

The paper is organized as follows. In Section \ref{sect:prel} we explain the  preliminary results used in the paper. In particular, we introduce the notions of complexes with non-positive immersion, staggered and  reducible complexes and a variation, bireducible complexes. We also review the principal facts about the Hughes-free division rings. In Section \ref{sect:cohgral} we develop a tool to prove the coherence of group algebras and homological coherence. As applications we prove Theorem \ref{npi} and show in Theorem \ref{mappingtorus} that the group algebra of an ascending  HNN-extension of a free group is coherent. Section \ref{sect:homcoh} is devoted to the proof of Theorem \ref{equivalence}. We finish this section with the proof of the first part of Theorem \ref{main}. In Section \ref{sect:flat} we construct several flat modules which allow us to control different $\Tor$ functors. In particular, we prove  Theorem \ref{weak} and, combining it with results of Section \ref{sect:cohgral}, we obtain Theorem \ref{th:coherentgroupalgebra} and  the second part of Theorem \ref{main}. Section \ref {sect:appl} contains further applications. In particular,  there we prove Theorems \ref{coxeter}, \ref{staggered_coherent} and \ref{HN}. We finish the paper with Section \ref{sect:final} where we discuss possible extensions of our results.

\section*{Acknowledgments}

The work of the first author  is partially supported by the grant  PID2020-114032GB-I00 of the Ministry of Science and Innovation of Spain  and by the ICMAT Severo Ochoa project  CEX2019-000904-S4. The work of the second author has received funding from the European Research Council (ERC) under the European Union's Horizon 2020 research and innovation programme (Grant agreement No. 850930).

We would like to thank Sam Fisher for pointing us to Proposition \ref{vanishing2nd} which simplified several of our previous arguments. We would like to thank Dawid Kielak for his valuable comments on the material in \cref{sect:homcoh}. We would also like to thank Jack Button for providing a simplification of the proof of \cref{quotient_tree} and Thomas Delzant for pointing out the work of Dicks--Dunwoody \cite{dicks_89}. We would like to thank the anonymous referees whose suggestions have improved this article.

\section{Preliminaries} \label{sect:prel}

\subsection{General notations}
All rings in this paper are associative and have an identity element.  All   ring homomorphisms send  the identity   to the identity. Generally, unless stated otherwise, we shall use $K$ to denote a field, $k$ to denote a commutative ring and $R$ to denote an arbitrary ring. By an {\bf $R$-ring} we understand a  ring homomorphism $\varphi: R\to S$. We will often refer to $S$ as an $R$-ring and omit the homomorphism $\varphi$ if $\varphi$ is clear from context.
Two $R$-rings $\varphi_1:R\to S_1$ and $\varphi_2:R\to S_2$ are said to be {\bf isomorphic}  if there exists a ring isomorphism $\alpha: S_1\to S_2$ such that $\alpha\circ \varphi_1=\varphi_2$.

For a ring $R$, a left $R$-module $M$ and $a\in R$, we put $$\Ann_M\left (a\right ) =\{m\in M\colon a\cdot m=0\}.$$

 Let $G$ be a group and $k$ a commutative ring with 1. We denote by $I_{k[G]}$ the augmentation ideal of $k[G]$. If $H$ is a subgroup of $G$ we denote by ${}^GI_{k[H]}$ the left ideal of $k[G]$ generated by $I_{k[H]}$ and by $I_{k[H]}^G$ the right ideal of $k[G]$ generated by $I_{k[H]}$.  Recall the following standard result (see, for example, \cite[Lemma 2.1]{Ja24}).
 
 \begin{lem} \label{augmentation} Let $H\le T$ be      subgroups of a group $G$ and $k$ a commutative ring with 1. Then the following holds.
\begin{enumerate}
\item The canonical map $$  I_{k[H]}\otimes_{k[H]} k[G]\to I_{k[H]}^G$$ sending  $a\otimes b$ to $ab$, is an isomorphism of right $k[G]$-modules.

\item  The canonical map  $ \left (I_{k[T]}/I_{k[H]}^T\right )  \otimes_{k[T]} k[G]\to I_{k[T]}^G/I_{k[H]}^G $, which sends\newline $ \left (a+I_{k[H]}^T\right ) \otimes b$ to $ab+I_{k[H]}^G$, is an isomorphism of $k[G]$-modules.
\end{enumerate}
\end{lem}

Given two left $k[G]$-modules $M$ and $N$,  $M\otimes_k N$ becomes a left $k[G]$-module if we define $g\left (m\otimes  n\right ) =gm\otimes  gn$.

\subsection{Projective, global and weak  dimensions}

Let $R$ be a ring and $M$ a left $R$-module. We say that $M$ is of {\bf projective dimension} at most $n$ if $M$ has a projective resolution of length $n$. A ring $R$ is of (left)  {\bf global dimension} at most $n$, if every left $R$-module has a projective resolution of length $n$. We will often use the following result about  global dimensions of group algebras.

\begin{pro}\label{gldim}
Let $G$ be a group of cohomological dimension $n$. Then for any field $K$, the ring $K[G]$ is of global dimension at most $n$.
\end{pro}

\begin{proof}
Since $G$ is of cohomological dimension $n$, there exists an exact sequence of left projective $\Z[G]$-modules
$$0\to P_n\to\ldots \to P_1\to P_0\to \Z\to 0.$$
Let $M$ be a left $K[G]$-module. After applying $\_ \otimes _{\Z} M$ we obtain the following exact sequence of $K[G]$-modules.
$$0\to P_n\otimes _{\Z} M\to\ldots \to P_1\otimes _{\Z} M\to P_0\otimes _{\Z} M\to M\to 0.$$
If $P$ is a left projective $\mathbb{Z}[G]$-module, then there exists a left $\mathbb{Z}[G]$-module $Q$ such that  
$P \oplus Q \cong \bigoplus_{i \in I} \mathbb{Z}[G]$.  
Thus, $(P \otimes_{\mathbb{Z}} M) \oplus (Q \otimes_{\mathbb{Z}} M) \cong \bigoplus_{i \in I} (\mathbb{Z}[G] \otimes_{\mathbb{Z}} M)$.  
Fixing a $K$-basis $\{m_j\}$ of $M$, we see that $1 \otimes m_j$ is a $K[G]$-basis of $\mathbb{Z}[G] \otimes_{\mathbb{Z}} M$, and so, $P \otimes_{\mathbb{Z}} M$ is projective.  
Thus, $P_i \otimes_{\mathbb{Z}} M$ are projective, and the ring $K[G]$ is of global dimension at most $n$.
\end{proof}

 The {\bf weak dimension} of  a right $R$-module $M$  is the largest $i$ for which there exists a left $R$-module $N$ such that $\Tor^R_i\left (M,N\right ) \ne 0$.

\subsection{Hughes-free,  Linnell and Dubrovin division rings}\label{dkg}
Let $G$ be a group. In this subsection we will assume that $K$ is a field and to simplify the exposition   we will  only consider the group algebra $K[G]$. However, we want to underline that all the definitions and results can be easily extended to the case of crossed products $E*G$, where $E$ is  a division ring.

Let $\phi: K[G]\to \D$ be a  division $K[G]$-ring. Let $N\le H\le G$ be subgroups of $G$. We denote by $\D_H$ the division closure of $\phi(K[H])$ in $\D$. We say that $\D$ is (left) {\bf $(N,H)$-free} if the  map $$\D_N\otimes_{K[N]} K[H]\to \D,\ d\otimes   a\mapsto d\phi(a),$$ is injective. An alternative reformulation of $(N,H)$-freeness is the following: if $q_1,\ldots, q_n\in H$ are in  different right $N$-cosets, then for any non-zero elements  $d_1,\ldots, d_n\in \D_N$, the sum  $\sum_{i=1}^n d_i\phi(q_i)$ is not equal to zero. It is clear that if $N\le H_1\le H_2\le G$ and $\D$ is $(N,H_2)$-free, then it is also $(N,H_1)$-free.

This property appeared in the work of Hughes \cite{Hu70} in the context of locally indicable groups. Let $G$ be a locally indicable group and $\phi: K[G]\to \D$ a  division $K[G]$-ring. We say that  $\phi: K[G]\to \D$ (or simply $\D$, when $\phi$ is clear from context) is {\bf Hughes-free} if $\D$ is epic (i.e. $\D=\D_G$) and $(N,H)$-free for any pair $N\unlhd H$ of subgroups  of $G$ with $H/N\cong \Z$. A very important contribution of Hughes is the following result proved in \cite{Hu70} (see also, \cite{DHS04} and \cite[Theorem 5.2]{JL20}).
\begin{teo}[Hughes]\label{hughes}
Let $K$ be a field and $G$  a locally indicable group. Then  up to $K[G]$-isomorphism  there exists at most one Hughes-free division $K[G]$-ring.
\end{teo}
In view of this result, if $G$ is locally indicable and the Hughes-free division $K[G]$-ring exists, we will denote it by $\D_{K[G]}$.
It was  conjectured that, for a given locally indicable group $G$ and a field K, $\D_{K[G]}$ always exists and this was proven in \cite[Corollary 6.7]{JL20} in the case where $K$ is of characteristic 0. The reader can consult \cite{Ja21} to see what is known  at this moment about the general case.

In order to generalize the notion of Hughes-free division ring to an arbitrary torsion-free group, Linnell proposed in \cite{Li06} the notion of strongly Hughes-free division $K[G]$-ring. Following Linnell, we say that an epic division $K[G]$-ring $\phi\colon K[G]\to \D$ is {\bf strongly Hughes-free} if it is $(N,H)$-free for any pair $N\unlhd H$ of subgroups  of $G$.

There is another instance where   $(N, H)$-freeness appears. An equivalent formulation of the strong Atiyah conjecture over $\Q$ for torsion-free groups $G$ says that the division closure $\D(G)$ of $\Q[G]$ in the ring of affiliated operators $\mathcal U(G)$ is a division ring (see, \cite[Proposition 1.2]{LS12}). The ring $\D(G)$ is called a {\bf Linnell}   ring. By  the discussion after \cite[Problem 4.5]{Li06},  the Linnell  ring (if it is a division ring) is $(N,G)$-free for any     subgroup  $N$ of $G$. In particular, by Theorem \ref{hughes}, if $G$ is locally indicable, $\D(G)\cong \D_{\Q[G]}$ as $\Q[G]$-rings. One consequence of this fact is that for locally indicable groups the $L^2$-Betti numbers  can be computed in a purely algebraic way: $$b_k^{(2)}(G)=\dim_{\D_{\Q[G]}} \Tor_k^{\Q[G]} (\D_{\Q[G]}, \Q).$$
This also leads us to the following definition. Let $G$ be a torsion-free group. We say that an epic division $K[G]$-ring $\phi\colon K[G]\to \D$ is {\bf Linnell} if it is $(N,G)$-free for any subgroup $N$  of $G$. In view of the previous discussion it is tempting to propose the following variation of the strong Atiyah conjecture for torsion-free groups.
\begin{Conj}
Let $K$ be a field and $G$ a torsion-free group. Then a Linnell division $K[G]$-ring exists and is unique up to $K[G]$-isomorphism.
\end{Conj}
The strong Atiyah conjecture indicates that in the case where $K$ is a subfield of $\CC$, a candidate for a Linnell division $K[G]$-ring is the division closure of $K[G]$ in $\mathcal U(G)$. If $K$ is an arbitrary field and $G$ has a right-invariant order, Dubrovin proposed such a candidate \cite{Du87} (see also \cite{Be18, GS15}). 

 Let $G$ be a group with a right-invariant order $\preceq$.  {\bf The  space of  Malcev-Neumann series} $\mathcal{MN}_{\preceq}(K[G])$  is the  abelian group consisting of formal infinite sums $m=\sum_{g \in G} k_gg$, with $k_g \in K$, such that the support of $m$, $$\supp \left (\sum_{g \in G} k_g g  \right )=\{g \in G:\ k_g \ne 0\}$$ is a well-ordered subset of $G$. 
 We denote by $\mathcal E_\preceq (K[G])$ the ring of endomorphisms of the abelian group $\mathcal{MN}_\preceq (K[G])$ and we will use the notation where the elements of  $\mathcal E_\preceq (K[G])$ act on $\mathcal{MN}_\preceq (K[G])$ on  the right.
 
  Since the order $\preceq$ is right-invariant, the ring $K[G]$ is embedded into $\mathcal E_\preceq (K[G])$ and we will identify the elements of  $K[G]$  with their images. The {\bf Dubrovin} ring $\D_{\preceq}(K[G])$ is the division closure of $K[G]$ in $\mathcal E_{\preceq} (K[G])$. Dubrovin conjectured that this ring is a division ring. The following result has appeared in the literature  in  slightly different contexts and we include its proof for the readers' convenience.
\begin{pro} \label{dubrovinlinnell}
Let $K$ be a field and   $G$ a group with a right-invariant order $\preceq$. 
 If the Dubrovin ring  $\D_{\preceq}(K[G])$ is  a division ring, then it is also a Linnell division $K[G]$-ring.
 \end{pro}
  \begin{proof} Let us denote  $\D_{\preceq}(K[G])$ by $\D$. Let $N $ be a subgroup of $G$. We want to show that if $q_1,\ldots, q_n\in G$ are   in  different right $N$-cosets, then for any non-zero elements  $d_1,\ldots, d_n\in \D_N$, the sum  $\sum_{i=1}^n d_i q_i\in \D$ is not equal to zero. Assume that $\sum_{i=1}^n d_i q_i\in \D=0$. Without loss of generality we can also assume that $q_1=1$.
  
  Given    a subset $T$ of $G$, we denote by $\pi_T$ the element of $\mathcal E_\preceq (K[G])$ defined by means of
$$\left(\sum_{g\in G} k_g g \right)\cdot \pi_T=\sum_{g\in T} k_gg\ (k_g\in K).$$
Observe that elements of $K[N]$ commute with $\pi_N$. Hence, the elements of $\D_N$ commute with $\pi_N$ too. Thus, for each $i=1,\ldots, n$,
$$\supp(1\cdot d_iq_i)=\supp (1\cdot \pi_N d_iq_i)=\supp (1\cdot d_i\pi_Nq_i)\le Nq_i.$$ Therefore,
$$1\cdot d_1=\left(1\cdot \sum_{i=1}^n d_iq_i \right)\pi_N=0.$$
Since $d_1$ is invertible in $\mathcal E_{\preceq}(K[G])$, we obtain a contradiction.
 \end{proof}

As we have mentioned above, if $G$ is locally indicable and $K$ is of characteristic zero, then the Hughes-free division ring $\D_{K[G]}$ is Linnell. Gr\"ater proved in \cite{Gr20} the same result for an arbitrary field $K$ if $\D_{K[G]}$ exists. He used the following characterization of locally indicable groups: a group is locally indicable if and only if it admits a Conradian order. We say that an order $\preceq $ on $G$ is {\bf Conradian} if for all  elements $f,g\succeq 1$ of $G$, there exists a natural number $n$ such that $g^nf\succ g$. 
\begin{teo}[Gr\"ater]\label{grater}
Let $K$ be a field and $G$ a locally indicable group. Let $\preceq$ be a Conradian order on $G$. If $\D_{K[G]}$ exists, then the Dubrovin ring $\D_{\preceq}(K[G])$ is a division ring. In particular, $\D_{K[G]}$ is the unique Linnell division $K[G]$-ring.
\end{teo}
\begin{proof}
By \cite[Corollary 8.3, Theorem 8.1]{Gr20}, $\D_{\preceq}(K[G])$ is a division ring and is isomorphic to $\D_{K[G]}$. By Proposition \ref{dubrovinlinnell}, $\D_{\preceq}(K[G])$  is Linnell.  Theorem \ref{hughes} implies that  $\D_{K[G]}$ is the unique Linnell division $K[G]$-ring.
\end{proof}

\subsection{Non-positive and not too positive immersions}

A two-complex is a two-dimensional CW-complex. Maps between two-complexes will always be assumed to be combinatorial. That is, $n$-cells map homeomorphically to $n$-cells. We will also always assume that attaching maps of two-cells are given by immersions. In particular, we are not allowing two-cells with boundary a point. If $\Gamma$ is a graph and $\lambda\colon\mathbb{S}\immerses \Gamma$ is an immersion of a disjoint union of circles, we denote by $X = (\Gamma, \lambda)$ the two-complex obtained by attaching a two-cell to $\Gamma$ along the image of each component of $\mathbb{S}$.

A two-complex $X$ has {\bf non-positive immersions} if for every immersion $Y\immerses X$ where $Y$ is a compact connected two-complex, we either have $\chi(Y)\leqslant 0$, or $\pi_1(Y) = 1$. A useful consequence of having non-positive immersions is the following, due to Wise \cite{wise_20}.

\begin{teo}
\label{npi_indicable}
If $X$ has non-positive immersions, then $\pi_1(X)$ is locally indicable.
\end{teo}

A variation of non-positive immersions is known as {\bf weak non-positive immersions}. We say $X$ has weak non-positive immersions if for every immersion $Y\immerses X$ with $Y$ compact and connected, we have $\chi(Y)\leqslant 1$.

\begin{pro}
\label{wnpi}
If $X$ is a two-complex with non-positive immersions and with $\pi_1(X) \neq 1$, then $X$ has weak non-positive immersions.
\end{pro}

\begin{proof}
Suppose for a contradiction that $X$ does not have weak non-positive immersions and so there is an immersion $Y\immerses X$ with $Y$ finite, connected, with $\pi_1(Y) = 1$ and $\chi(Y)\geqslant 2$. Since $Y$ is connected, we may assume that $X$ is connected. Since every connected one-dimensional cell complex has Euler characteristic at most one, we may assume that $X$ and $Y$ have at least one two-cell. If some one-cell $e\subset Y$ is not traversed by an attaching map of a 2-cell, then each component of $Y - e$ also satisfies these assumptions. Hence, we may assume that every one-cell in $Y$ is traversed by at least one attaching map of a two-cell. Denote by $\overline  {X}$ the image of $Y$ in $X$.

Suppose first that $X^{(1)}$ deformation retracts to $\overline  {X}^{(1)}$ and that the map $Y^{(1)}\to \overline  {X}^{(1)}$ is a cover. Note that since attaching maps of two-cells are immersions, the fact that $X^{(1)}$ deformation retracts to $\overline  {X}^{(1)}$ implies that all two-cells are supported in $\overline{X}^{(1)}$ and so $\pi_1(\overline{X})\to \pi_1(X)$ is a surjection. Since $Y$ was finite, $Y^{(1)}\to \overline{X}^{(1)}$ is a finite cover and so $\pi_1(Y)$ maps onto a finite index subgroup of $\pi_1(X)$. Since $\pi_1(Y) = 1$, it follows that $\pi_1(X)$ must be finite and so $\pi_1(X) = 1$ as $\pi_1(X)$ is locally indicable by \cref{npi_indicable}.

Now suppose that $X^{(1)}$ deformation retracts to $\overline  {X}^{(1)}$, but that the map $Y^{(1)}\to \overline  {X}^{(1)}$ is not a cover. Then there is some point $v\in Y^{(0)}$ and a one-cell $\overline  {e}$ in $\overline  {X}$, adjacent to the image $\overline  {v}$ of $v$, such that there is no one-cell adjacent to $v$ that maps to $\overline  {e}$. Hence, we may attach a one-cell $e$ along an endpoint to $Y$, obtaining $Y' = Y\cup_v e$, and extend our immersion $Y\immerses \overline  {X}$ to $Y'\immerses\overline  {X}$ by mapping $e$ to $\overline  {e}$. Now consider the graph $\Gamma = \overline  {X}^{(1)} - \overline  {e}$. If $\Gamma$ is connected and simply connected, then $X^{(1)}$ deformation retracts onto a copy of $S^1$. Since attaching maps of two-cells are immersions, it follows that all attaching maps of two-cells in $X$ must be supported in this copy of $S^1$. Thus, since $X$ has at least one two-cell, we have that $\pi_1(X)$ is a proper quotient of $\pi_1(X^{(1)}) = \Z$ and so $\pi_1(X) = 1$ as $\pi_1(X)$ is locally indicable. Now suppose that $\Gamma$ is connected, but not simply connected. Then there is a connected subgraph $\overline{S}\subset \overline{X}^{(1)}$ containing the other endpoint of $\overline{e}$, call it $\overline{u}$ (which may be equal to $\overline{v}$), and with $\chi(\overline{S})= 0$. Now define $Y''$ to be the two-complex obtained from $Y'$ by attaching a copy $S$ of $\overline{S}$ to the other endpoint of $e$ along $u\in S$, the corresponding copy of $\overline{u}$. The resulting map $Y''\immerses \overline  {X}$ is an immersion and it is not hard to see that $\chi(Y'')\geqslant 1$ and $\pi_1(Y'')\cong \Z$, contradicting the hypothesis that $X$ has non-positive immersions. Finally, suppose that $\Gamma$ is not connected and denote by $\Gamma'$ the component not containing $v$. Since there is a two-cell in $\Gamma'$ whose attaching map traverses $\overline  {e}$, it follows that $\Gamma'$ is not simply connected. Here we are using the fact that all attaching maps of two-cells are immersions. From here we may apply the same trick as above to obtain a contradiction. 

If $X^{(1)}$ does not deformation retract to $\overline{X}^{(1)}$, then this means that there is some non-trivial immersed path in $X^{(1)}$, intersecting $\overline{X}^{(1)}$ only at its endpoints, say $v$ and $w$ (these could be equal). We may assume that the image of this path is a connected subgraph $\Gamma$ with $\Gamma \cap \overline{X} = \{v, w\}$. If $v \neq w$, then we may assume that this path is embedded and so $\chi(\Gamma) = 1$. If $v = w$, then we may assume that $\Gamma$ is an embedded path leading out from $v$ and connected to an embedded cycle so that $\chi(\Gamma) = 0$. Now since $v, w \in \overline{X}$, there are vertices $v', w' \in Y$ (equal if $v, w$ are equal) that map to $v$ and $w$ respectively. Let $Y'$ be obtained from $Y$ by gluing a copy of $\Gamma$ along $v'$ and $w'$. If $v' = w'$, we have $\chi(\Gamma)= 0$ and so $\chi(Y') = \chi(Y) + \chi(\Gamma) - 1=\chi(Y)-1$. If $v'\neq w'$, we have $\chi(\Gamma)= 1$ and so $\chi(Y') = \chi(Y) + \chi(\Gamma) - 2=\chi(Y) - 1$. The immersion $Y \immerses X$ now extends in the obvious way to an immersion $Y' \immerses X$, but now $\chi(Y') = \chi(Y) - 1\geqslant 1$ and $\pi_1(Y') = \Z$. Hence $X$ could not have had non-positive immersions.
\end{proof}

Combining \cref{wnpi} with \cite[Theorem 1.6]{wise_20}, we have the following.

\begin{cor}
\label{npi_aspherical}
Let $X$ be a two-complex with non-positive immersions. Then either $X$ is aspherical or $\pi_1(X) = 1$.
\end{cor}

If $X$ has non-positive immersions, it turns out that we may derive explicit bounds on the $L^2$-Betti numbers of finitely generated subgroups of $\pi_1(X)$. The proof of this result below closely follows the proof of a similar result due to Louder--Wilton \cite[Corollary 1.6]{LW17}. We include a proof for completeness.

\begin{pro}
\label{l2betti}
Let $X$ be a two-complex with non-positive immersions. If $H$ is a finitely generated subgroup of $\pi_1(X)$, then 
\[
b_2^{(2)}(H) \leqslant b_1^{(2)}(H).
\]
\end{pro}

\begin{proof}
Assume that $H$ is non-trivial. By \cite[Lemma 4.4]{louder_20}, there is a sequence of $\pi_1$-surjective immersions of finite connected two-complexes
\[
Y_0\immerses Y_1\immerses\ldots\immerses Y_i\immerses\ldots\immerses X
\]
such that
\[
H = \varinjlim \pi_1(Y_i).
\]
Put $H_i=\pi_1(Y_i)$. Then there exists a finitely generated free group $F$ and normal subgroups $N\ge N_{i+1}\ge N_i$ ($i\in \N$) of $F$ such that
\[
H\cong F/N,\ H_i\cong F/N_i \textrm{\ and\ } N=\cup_{i\in \N} N_i.
\]
In particular, we have that
\[
N_{\text{ab}}\cong \varinjlim (N_i)_{\text{ab}}
\]
viewed as $\Q[F]$-modules. 

The $\Q[H]$-module $N_{\text{ab}}$ is the relation module arising in the exact sequence
$$0\to N_{\text{ab}}\to \Q[H]^{d(F)}\to \Q[H]\to \Q\to 0,$$ see \cite[Proposition 5.4]{Br82}. Moreover, $N_{\ab}$ is projective since $H$ has cohomological dimension at most two. Therefore, we obtain that
\begin{align*}
b_2^{(2)}(H)-b_1^{(2)}(H)&=\dim_{\D_{\Q[H]}} H_2(H; \D_{\Q[H]})-\dim_{\D_{\Q[H]}} H_1(H; \D_{\Q[H]})\\ 
					  &=\dim_{\D_{\Q[H]}} \left (\D_{\Q[H]}\otimes_{\Q[F]} N_{\text{ab}}\right )-d(F)+1.\end{align*}
Thus, we  have that 
\begin{align*}
b_2^{(2)}(H)-b_1^{(2)}(H)&=\dim_{\D_{\Q[H]}} \left (\D_{\Q[H]}\otimes_{\Q[F]} N_{\text{ab}}\right )-d(F)+1\\ 
					&\leqslant \sup_{i\in \N} \dim_{\D_{\Q[H]}} \left (\D_{\Q[H]}\otimes_{\Q[F]} (N_i)_{\text{ab}}\right )-d(F)+1\\
					&=\sup_{i\in \N} \dim_{\D_{\Q[H]}} H_2(H_i; \D_{\Q[H]})-\dim_{\D_{\Q[H]}} H_1(H_i; \D_{\Q[H]}).
\end{align*}
 Since $X$ has non-positive immersions, we see that $\chi(Y_i)\leqslant 0$ for all $i$. Since $Y_i$ is aspherical for all $i$ by \cref{npi_aspherical}, we have
\[
\chi(Y_i) = \dim_{\D_{\Q[H]}} H_2(H_i; \D_{\Q[H]})-\dim_{\D_{\Q[H]}} H_1(H_i; \D_{\Q[H]}).
\]
This finishes the proof.
\end{proof}

Louder--Wilton introduced a slight generalisation of non-positive immersions in \cite{LW17}. If $C_k$ is the standard presentation complex of $\Z/k\Z$ and $C_{k, l}$ is the $l$-fold cover of $C_k$, then $X$ has {\bf not-too-positive immersions (NTPI)} if for every immersion $Y\immerses X$ where $Y$ is a compact connected two-complex, we have that $Y$ is homotopy equivalent to a wedge of subcomplexes of $C_{k, l}$'s and a subcomplex $Z\subset Y$ with $\chi(Z)\leqslant 0$. The main result of \cite{LW17} established NTPI for presentation complexes of all one-relator groups.

Louder--Wilton \cite{LW17} prove the proposition below for two-complexes with a single two-cell. Their proof only uses the fact that such a two-complex has NTPI and that the second integral homology group of finitely generated subgroups of one-relator groups must be torsion-free. Considering instead homology with rational coefficients, we may recover the following statement for all two-complexes with NTPI. Although the argument is identical to that of \cite[Corollary 1.6]{LW17}, we include a proof for the convenience of the reader.

\begin{pro}
\label{2betti}
Let $X$ be a two-complex with NTPI. If $H$ is a finitely generated subgroup of $\pi_1(X)$ that is not a free product of finite cyclic groups, then 
\[
b_2(H)\leqslant b_1(H) - 1.
\]
\end{pro}

\begin{proof}
Since $H$ is finitely generated, by Grushko's theorem $H$ admits a free product decomposition $H\isom H_1*\ldots*H_k*F$ where each $H_i$ is freely indecomposable and finitely generated and where $F$ is a finitely generated free group. Since $b_1, b_2$ are additive with respect to free products and $b_1(H) = b_2(H) = 0$ for $H$ finite cyclic, it suffices to prove the result assuming that $H$ is freely indecomposable and not finite or infinite cyclic.

Just as Louder--Wilton do in the discussion preceeding their proof of \cite[Corollary 1.6]{LW17}, we may produce $\pi_1$-surjective immersions 
\[
Y_0\immerses Y_{1}\immerses \ldots \immerses Y_n\immerses\ldots \immerses Y\immerses X
\]
where $Y\immerses X$ is the cover with $\pi_1(Y) = H$ and with $\varinjlim Y_i = Y$. By definition of NTPI, each $Y_i$ is homotopy equivalent to a subcomplex $Z_i\subset Y_i$ with $\chi(Z_i)\leqslant 0$ wedged with finitely many copies of $C_{k, l}$ for various values of $k, l$. By Delzant's version of the well-known Scott Lemma, as proved by Swarup \cite[Theorem 2.3]{Sw04}, since $H$ is freely indecomposable, there is some $K\geqslant 0$ such that $\pi_1(Y_i)$ is freely indecomposable for all $i\geqslant K$. By possibly selecting a larger $K$, since $\pi_1(Y_i)\to \pi_1(Y_{i+1})\to H$ are surjective, we have $b_1(H) \leqslant b_1(Y_{i+1})\leqslant b_1(Y_i)$ for all $i\geqslant K$. Thus, by possibly replacing $K$ with a larger integer, we may also assume that $b_1(H) = b_1(Y_i)$ for all $i\geqslant K$. By removing the first $K$ terms in the sequence, we thus may assume that $Y_i = Z_i$ and $b_1(Y_i) = b_1(H)$ for all $i$. Since $\chi(Y_i) = 1 - b_1(Y_i) + b_2(Y_i)$ we have $b_2(Y_i)\leqslant b_1(Y_i) - 1$ for all $i$. Note that $H_2(Y, \Q)$ quotients onto $H_2(H, \Q)$ since the chain complex $C_*(\widetilde{Y})$, where $\widetilde{Y}$ is the universal cover of $Y$, is a $\Z H$-complex which can be completed to a resolution of $\Z$. Finally, since $H_2(Y, \Q) = H_2(\varinjlim Y_i, \Q) = \varinjlim H_2(Y_i, \Q)$, we have that 
\[
b_2(H)\leqslant b_2(Y)\leqslant \sup_{i\in \N}b_2(Y_i)\leqslant b_1(H) - 1
\]
and the proof is complete.
\end{proof}
 
We remark that the statement of \cite[Corollary 1.6]{LW17} is missing the assumption that $H$ is not a free product of finite cyclic groups (or the trivial group). The proof is of course entirely correct with this assumption.

 \subsection{Staggered and reducible complexes}\label{sec:staggered}

A two-complex $X$ is {\bf staggered} if there is a total order on its two-cells and a subset of its one-cells satisfying the following:
\begin{enumerate}
\item The attaching map for each two-cell traverses at least one ordered one-cell.
\item If $\alpha<\beta$ are two-cells, then $\min{\alpha}<\min{\beta}$ and $\max{\alpha}<\max{\beta}$, where $\min{\delta}$ (respectively, $\max{\delta}$) is the smallest (respectively, largest) edge traversed by the attaching map for the two-cell $\delta$.
\end{enumerate}
A group has a {\bf staggered presentation} if it has a presentation complex that is staggered. Groups with staggered presentations are a class of groups that arise naturally in the study of one-relator groups, see \cite{lyndon_01}.

If $X = (\Gamma, \lambda)$ is a two-complex, an edge $e\subset \Gamma$ is a {\bf reducing edge} (respectively, {\bf collapsing edge}) if $\lambda^{-1}(e)$ is contained in a single component (respectively, a single edge). If $e\subset \Gamma$ is a reducing edge (respectively, a collapsing edge), we call the two-complex $Z$ obtained from $X$ by removing $e$ and the unique two-cell whose attaching map traverses $e$, the {\bf reduction along $e$} (respectively, the {\bf collapse along $e$}). In this case we will say that $Z\subset X$ is a {\bf reduction} ({\bf collapse}). A reducing edge $e\subset \Gamma$ is {\bf strictly involved} if, denoting by $c\subset X$ the unique two-cell whose attaching map traverses $e$, the attaching map for $c$ is not freely homotopic within $X - c$ into its reduction $X - (e\cup c)$.

A two-complex is {\bf reducible} (respectively, {\bf collapsible}) if every subcomplex containing at least one two-cell has a reducing edge (respectively, collapsing edge) which is strictly involved.  A group has a {\bf reducible presentation} if it has a presentation complex that is reducible. This generalises staggered complexes and presentations, see \cref{injective_sub}.

Reducible two-complexes were first introduced by Howie \cite{Ho82} and shown to be aspherical and have locally indicable $\pi_1$ if without proper powers: if $X$ is reducible, then $X$ does not have {\bf proper powers} if   for each two-cell $\alpha\subset X$, the attaching map of $\alpha$ does not represent a proper power in $\pi_1(X - \alpha)$. Helfer--Wise established the non-positive immersions property for reducible complexes without proper powers in \cite[Corollaries 5.6 \& 7.6]{HW16}, stated below. A more general version is due to Howie--Short \cite{howie_20} and, independently, Millard \cite{millard_21}.

\begin{teo}
\label{reducible_npi}
If $X$ is a reducible complex without proper powers, then $X$ has non-positive immersions.
\end{teo}

If $H$ and $G$ are groups and $w\in G*H$ is an element that is not conjugate into $G$ or $H$, then we say $G*H/\normal{w}$ is a {\bf one-relator product} of $H$ and $G$. If $w$ is moreover not a proper power in $G*H$, then $G*H/\normal{w}$ is a {\bf one-relator product without proper powers}. With this terminology, the class of groups isomorphic to the fundamental groups of reducible two-complexes (without proper powers) can also be described as the smallest class $\mathcal R$ of groups containing all free groups that is closed under free products and one-relator products (without proper powers).

Let $F$ be a free group with respect to free generators $X$ and let $R=\{r_1,\ldots, r_n\}$ be a subset of $F$. We say that $R$ is {\bf strictly reducible with respect to $X$} if for some  $X_0\subset X$ and $\{x_1,\ldots, x_n\} \subset X$ we have that for each $1\le i\le n$,
\begin{enumerate}

\item $r_i\in  \langle X_i\rangle \cong \langle X_{i-1}\rangle *\langle x_i\rangle$, where $X_i=X_0\cup \{x_1,\ldots, x_i\} $;

\item for some $s\ge 1$, $r_i=a_1\cdot b_1\cdot \ldots \cdot a_s\cdot b_s$, where $a_j\in  \langle X_{i-1}\rangle $ and $b_j\in \langle x_i\rangle$ and the image of each $b_i$ in the group $G_i=\langle X_{i-1}\rangle /\normal{r_1,\ldots, r_{i-1}}$ is non-trivial;

\item the image of  $r_i$ in $G_{i-1}*\langle x_i\rangle$ is not a proper-power.
\end{enumerate}

Observe that if  the two-complex associated to a presentation $\langle X|R\rangle$ is reducible without proper powers and $R$ finite then we can modify $R$ and obtain    $R^\prime\subset F(X)$ such that $\langle X|R\rangle\cong  \langle X|R^\prime \rangle$ and $R^\prime$ is strictly reducible with respect to $X$. We argue by induction on $|R|$. Assume we have constructed  
$\{r_1^\prime, \ldots, r_{n-1}^\prime\} \subset \langle X \rangle$ such that  
$\{r_1^\prime, \ldots, r_{n-1}^\prime\}$ is strictly reducible with respect to $X$ and  
$G_{n-1} = \langle X \mid r_1, \ldots, r_{n-1} \rangle \cong \langle X \mid r_1^\prime, \ldots, r_{n-1}^\prime \rangle$.  

Since $\langle X \mid R \rangle$ is reducible, there exists $x_n \in X$ that appears only in $r_n$. Put  
$X_{n-1} = X \setminus \{x_n\}$ and $H_{n-1} = \langle X_{n-1} \mid r_1^\prime, \ldots, r_{n-1}^\prime \rangle$.  
Thus, we have that $G_{n-1} \cong H_{n-1} * \langle x_n \rangle$.  

Let $\overline{r_n}$ be the image of $r_n$ in $G_{n-1}$. If $\overline{r_n} \in \langle x_n \rangle$, then since it is not a proper power,  
$\overline{r_n} = x_n^{\pm 1}$, and we can take $r_n^\prime = x_n$. On the other hand, since the presentation complex for $\langle X \mid R \rangle$ is reducible,  
$\overline{r_n} \notin H_{n-1}$.  
Thus, some conjugate of $\overline{r_n}$ can be represented as  
\[
\overline{r_n} = \overline{a_1} \cdot b_1 \cdot \ldots \cdot \overline{a_s} \cdot b_s,
\]
where  $s\ge 1$, $1 \neq \overline{a_j} \in H_{n-1}$ and $1 \neq b_j \in \langle x_n \rangle$.  
Now, $r_n^\prime$ is constructed by lifting $\overline{a_j}$ to $\langle X_{n-1} \rangle$.

We will say that a subgroup $W$ of a free group $F$ is {\bf strictly reducible} if there is a set of free generators $X$ of $F$ and a strictly reducible with respect to $X$ subset $R$ of $F$ such that $W$ is generated by $R$.

We conclude this section with the following observation. Let $X$ be a reducible two-complex (without proper powers) and let $Y$ be the two-complex obtained from $X$ by adding a vertex and attaching an edge joining each vertex of $X$ with this new vertex. We have $\pi_1(Y)\cong \pi_1(X)*F$ where $F$ is some free group. By collapsing all the new edges to the new vertex we obtain a new reducible two-complex (without proper powers) with a single vertex. Hence $\pi_1(X)*F$ has a reducible presentation (without proper powers).

\subsection{Bireducible complexes}

In order to facilitate the proof of  \cref{staggered_coherent}, we shall need to introduce a new class of two-complexes that lie in between staggered and reducible two-complexes. We say a two-complex $X$ is {\bf bireducible} if every subcomplex $Z\subset X$ containing at least two two-cells, has at least two reducing edges associated to distinct two-cells. 

The aim of this section is to show that such two-complexes have NTPI. We first require a version of the classic Magnus Freiheitssatz \cite{magnus_30} for bireducible complexes.

\begin{lem}
\label{injective_sub}
If $X$ is a finite bireducible two-complex, then the following hold:
\begin{enumerate}
\item If $Z\subset X$ is a reduction, then the homomorphism $\pi_1(Z)\to \pi_1(X)$ induced by inclusion is injective for any choice of basepoint.
\item If $U\subset X$ is a subcomplex containing a single two-cell whose attaching map is surjective, then $\pi_1(U)\to \pi_1(X)$ is injective.
\end{enumerate}
In particular, $X$ is reducible.
\end{lem}

\begin{proof}
The proof is by induction on the number of two-cells in $X$. We prove the two statements at once. In the base case, the first statement is the Magnus' Freiheitssatz \cite{magnus_30} and the second statement is clear. Now suppose that $X$ contains more than one two-cell and assume the inductive hypothesis. As $X$ is bireducible, there is a second reduction $Y\subset X$ such that $X = Y\cup Z$ and $Y\cap Z$ is a reduction of $Y$ and $Z$. If $Z$ is not connected, then we attach an edge connecting the components. We do the same for $Y$ and $Y\cap Z$. Attaching edges modifies the fundamental group by possibly adding a free group as a free factor. By induction, $\pi_1(Y\cap Z)\to \pi_1(Y), \pi_1(Z)$ are injective and so we have
\[
\pi_1(X) \isom \pi_1(Y)*_{\pi_1(Y\cap Z)}\pi_1(Z)
\]
from whence the first statement follows. Now $U$ must either be contained in $Y$ or in $Z$ and so the second statement also follows by induction.

In order to complete the proof we just need to show that for every reduction $Z\subset X$, the reducing edge $e\subset X - Z$ is strictly involved. If $U\subset X$ is the smallest subcomplex containing the two-cell in $X - Z$, then the second statement we proved above implies that we just have to prove this for $U$. But this is Magnus' Freiheitssatz once more.
\end{proof}

Let $\Lambda$ be a connected subgraph of a bireducible complex $X$. We say that $\Lambda$ is {\bf small} if for every subcomplex $Y$ of $X$  containing $\Lambda$, $Y$ is either a graph or there exists a reduction $Z\subset Y$ such that $\Lambda\subseteq Z$. By \cref{injective_sub}, the map $\pi_1(\Lambda)\to \pi_1(X)$ is injective when $X$ is finite and so we are justified in calling $\pi_1(\Lambda)$ a {\bf Magnus subgroup} of $\pi_1(X)$. In the language of \cref{injective_sub}, any proper subgraph $\Lambda\subset U$ is small. Indeed, let $Y$ be a subcomplex of $X$ containing $\Lambda$ and assume that it is not a graph.  
If the two-cell of $U$ is the only two-cell of $Y$, then we can reduce any edge of $U$ not belonging to $\Lambda$. 
If $Y$ contains another two-cell, then we can make a reduction using an edge that is not in $U$.

An immersion $\lambda\colon S^1\immerses \Gamma$ of a cycle is {\bf primitive} if it does not factor through a proper cover $S^1\immerses S^1$. If $\lambda\colon\mathbb{S}\immerses \Gamma$ is an immersion of a disjoint union of circles, we call it primitive if the restriction to each component is primitive. 

\begin{lem}
\label{proper_powers}
Let $X = (\Gamma, \lambda)$ be a finite bireducible two-complex. Then $X$ is a reducible two-complex without proper powers if and only if $\lambda$ is primitive.
\end{lem}

\begin{proof}
Let $S^1\immerses \Gamma$ be the attaching map of one of the two-cells $\alpha$. By \cref{injective_sub}, if $U\subset X$ is the smallest subcomplex containing $\alpha$, then $\pi_1(U)$ is a subgroup of $\pi_1(X)$. If $S^1\immerses \Gamma$ represents a proper power in $\pi_1(X-\alpha)$, then it would be homotopic to a cycle $S^1\immerses U^{(1)}$ that is not primitive. In particular, $\pi_1(U)$ would have torsion. By the characterisation of one-relator groups with torsion due to Karrass--Magnus--Solitar \cite{karrass_60}, $\pi_1(U)$ is torsion-free if and only if $S^1\immerses\Gamma$ is primitive.
\end{proof}

The following theorem follows by combining \cref{proper_powers} with \cite[Theorem 5.5 \& Corollary 7.6]{HW16}. Since the terminology they use is different, we include a proof to explain how the results are used explicitly.

\begin{teo}
\label{bireducible_immersion}
Let $\Gamma$ be a finite graph and let $\lambda\colon \mathbb{S}\immerses\Gamma$ be an immersion of primitive cycles so that $X = (\Gamma, \lambda)$ is a bireducible two-complex. Let $\Theta$ be a finite connected non-empty graph and $\theta\colon\Theta\immerses \Gamma$ an immersion. If $\mathbb{S}'$ denotes the union of the cycles of the pullback graph $\Theta\times_{\Gamma}\mathbb{S}$, then either $(\Theta, \mathbb{S}'\immerses\Theta)$ is collapsible, or
\[
|\pi_0(\mathbb{S}')|\le  -\chi(\Theta).
\]
\end{teo}

\begin{proof}
Let $\lambda = \sqcup \lambda_i$ where each $\lambda_i$ is an immersion of a single circle $S_i\isom S^1$. By giving labels to the edges in $\Gamma$, choosing a basepoint and orientation in $S^1$ and pulling back the labelling, we may assign a label $w_i$ to each $\lambda_i$. Helfer--Wise define a $w_i$-cycle in $\Theta\immerses\Gamma$ to be a cycle $S^1\immerses \Theta$ which, after pulling back labels from $\Gamma$ and choosing a basepoint and orientation on $S^1$, is labelled by $w_i^n$ for some $n\geqslant 1$. Two $w_i$-cycles are equivalent if there is an isomorphism between the two copies of $S^1$ making all the maps commute. Then Helfer--Wise define $\overline{\#}_{w_i}(\Theta)$ to be the number of equivalence classes of $w_i$-cycles in $\Theta$ and $\overline{\#}_w(\Theta) = \sum\overline{\#}_{w_i}(\Theta)$. Since each cycle $\lambda_i$ is primitive, it is readily seen that there is a bijection between the equivalence classes of $w_i$-cycles in $\Theta$ and the components of $\Theta\times_{\Gamma}S_i$ that are circles (note that every component is either a circle or contractible). Thus, since $\Theta\times_{\Gamma}\left(\sqcup S_i\right) = \sqcup\left(\Theta\times_{\Gamma}S_i\right)$, we have that $\overline{\#}_w(\Theta)$ is the number of components of $\Theta\times_{\Gamma}\left(\sqcup S_i\right)$ that are circles. In particular, $\pi_0(\mathbb{S}') = \overline{\#}_w(\Theta)$.

\cite[Corollary 7.6]{HW16} states that a finite reducible two-complex without proper powers is slim. We shall not need to define what it means to be slim as we are going to apply another result directly from \cite{HW16}. By \cref{proper_powers}, the two-complex $X = (\Gamma, \lambda)$ is a finite reducible two-complex without proper powers. Hence $X$ is slim.

\cite[Theorem 5.5]{HW16} states that if $X = (\Gamma, \lambda)$ is a slim two-complex, and $\Theta\immerses \Gamma = X^{(1)}$ is an immersion of graphs, then either 
\[
|\pi_0(\mathbb{S}')| = \overline{\#}_w(\Theta)\leqslant b_1(\Theta) - 1 = - \chi(\Theta),
\]
or $(\Theta, \mathbb{S}'\immerses\Theta)$ collapses to a tree. Hence, the proof is complete.
\end{proof}

\begin{rem*}
In the proof of \cref{bireducible_immersion} we showed that, in the language of \cite{HW16}, a bireducible complex without proper powers is slim. With a more careful proof it may also be shown that they are also bi-slim, a slightly stronger property. We will not need this fact for our purposes.
\end{rem*}

The proof of the following corollary is identical to that of \cite[Corollary 1.5]{LW17}, replacing the use of \cite[Theorem 1.2]{LW17} with \cref{bireducible_immersion}. We remark here that in \cite{LW17}, the term `reducible' is used to mean `has a collapsing edge' in our terminology.

\begin{cor}
\label{ntpi}
If $X$ is a bireducible two-complex, then $X$ has NTPI.
\end{cor}

\section{Criteria for the coherence of group algebras and homological coherence}
 \label{sect:cohgral}

The proof of the following result is a variation of an argument due to Kropholler--Linnell--L\"{u}ck \cite[Lemma 4]{KLL09}.

\begin{lem}
\label{fg}
Let $R$ be a ring and assume that $R$ can be embedded in a division ring $\D$. If $P$ is a projective $R$-module, then $P$ is finitely generated if and only if $\dim_{\D}(P\otimes_{R}\D)<\infty$.
\end{lem}

\begin{proof}
If $P$ is finitely generated then certainly $\dim_{\D}(P\otimes_{R}\D)<\infty$. Suppose now that $\dim_{\D}(P\otimes_{R}\D) =n<\infty$. Since $\D$ is a division ring, there are elements $m_1,\ldots, m_n\in P$ such that $$\D\otimes_R P=\sum_{i=1}^n \D\left (1\otimes  m_i\right ) .$$
 
 Since $P$ is projective, there exists a free left $R$-module $L$ such that $L=P_1\oplus P_2$ where $P = P_1$. Thus, any element $l\in L$ can be uniquely written as $l=p_1+p_2$, where $p_1\in P_1$ and $p_2\in P_2$. We define $\pi_{P_1}:L\to P_1$ by $\pi_{P_1}\left (l\right ) =p_1$.
 
  Put $N=\sum Rm_i$. Since $N$ is finitely generated, there exists a finitely generated free summand $L_1$ of $L$ that contains $N$. Consider the natural map
 $\tau: P_1/N\to L/L_1$. Since $\D\otimes_{R}\left (P_1/N\right ) =0$, $\D\otimes_{R} \im \tau=0$. Thus, since $ \im \tau\le  L/L_1$ is a submodule of a free $R$-module, $\im \tau =\{0\}$.
 This implies that $P_1\le L_1$, and so $$P_1=\pi_{P_1}\left (P_1\right ) =\pi_{P_1}\left (L_1\right ) $$ is finitely generated.
\end{proof}

\begin{pro}\label{fp}
Let $R$ be a ring and assume that $R$ can be embedded in a division ring $\D$. Let $M$ be a finitely generated  left $R$-module of projective dimension at most 1.
Then $\dim_{\D} \Tor_1\left (\D,M\right ) $ is finite if and only if $M$ is finitely presented.
 \end{pro}
 \begin{proof}
 Since $M$ is  a finitely generated  left $R$-module of projective dimension at most 1 there exists an exact sequence of left $R$-modules
 $$0\to P_1\to P_0\to M\to 0$$
 with $P_0$ and $P_1$ projective and $P_0$ finitely generated. The calculation of $\Tor_1^R\left (\D, M\right ) $ gives us the exact sequence of left $\D$-modules
 $$0\to  \Tor_1^R\left (\D, M\right ) \to \D\otimes_R P_1\to \D\otimes_R P_0.$$
Now \cref{fg} states that $P_1$ is finitely generated if and only if $\dim_{\D} \Tor_1\left (\D,M\right )$ is finite. Since $P_1$ is finitely generated if and only if $M$ is finitely presented, the proof is complete.
\end{proof}

\begin{cor}\label{crit}
Let $R$ be a ring and assume that $R$ can be embedded in a division ring $\D$. Assume that $R$ is of global dimension at most 2 and the  right  $R$-module $\D$ is of weak dimension at most 1.
Then $R$ is coherent.
\end{cor}

\begin{proof}
Let $I$ be a finitely generated  left ideal of $R$. Let $0\to P\to R^k\to I\to 0$  be an exact sequence of left $R$-modules. Hence we  obtain the exact sequence
$$0\to P\to R^k\to R\to R/I\to 0.$$
Since $R$ is of global dimension at most 2, the projective dimension of $R/I$ is at most 2. Hence, by \cite[Proposition 8.6(2)]{Rotbook}, $P$ is projective, and so $I$ is of projective dimension at most 1. Since $\Tor_1\left (\D,I\right ) =\Tor_2\left (\D, R/I\right ) $ and  
the  right  $R$-module $\D$ is of weak dimension at most 1, $\Tor_1\left (\D,I\right ) =0$. Thus, by Proposition \ref{fp}, $I$ is finitely presented, and so, $R$ is coherent.
\end{proof}
 As an application we show that the group algebras of ascending HNN-extensions of free groups are  coherent.
\begin{teo}\label{mappingtorus}
Let $K$ be a field and  $G$  an ascending  HNN-extension of a free group. Then the group algebra  $K[G]$ is coherent.
\end{teo}
\begin{proof}
Let us briefly explain the construction of $\D_{K[G]}$. For details  see, for example, \cite{Ja21}. 

The group $G$ has a locally free normal subgroup $N$ such that $G/N\cong \Z$. Let $t\in G$ be such that $G/N$ is generated by $tN$ and let $\tau:K[N]\to K[N]$ be the automorphism induced by the conjugation by $t$: $\tau\left (a\right ) =t at^{-1}$. Then $K[G]$ is naturally isomorphic to the ring of twisted Laurent polynomials $K[N][t^{\pm 1}, \tau]$.  Since $N$ is locally free, there exists $\D_{K[N]}$ (see \cite[Theorems 1.1 \& 3.7]{Ja21}). 
  Since $\D_{K[N]}$ is unique, we can   extend $\tau$ to an automorphism  $\D_{K[N]}\to \D_{K[N]}$, which we will also call $\tau$. Then $\D_{K[G]}$ is isomorphic to the classical Ore ring of fractions of $\D_{K[N]}[t^{\pm 1},\tau]$. 

Since $N$ is locally free ,   $\Tor_2^{K[N]}\left ( \D_{K[N]}, M\right ) =0$ for an arbitrary finitely presented  left $K[N]$-module $M$. However, $\Tor_2^{K[N]}\left ( \D_{K[N]}, -\right )$ commutes with direct limits. Therefore, $\Tor_2^{K[N]}\left ( \D_{K[N]}, M\right ) =0$ for an arbitrary left $K[N]$-module $M$.

Let $M$ be a left $K[G]$-module. By Shapiro's lemma we obtain
$$\Tor_2^{K[G]}\left ( \D_{K[N]}[t^{\pm 1},\tau], M\right ) =\Tor_2^{K[N]}\left ( \D_{K[N]}, M\right ) =0.$$
Thus, $\Tor_2^{K[G]}\left ( \D_{K[G]}, M\right ) =0$ as well. Hence the right $K[G]$-module $\D_{K[G]}$ is of weak dimension at most 1.

Since $G$ is a HNN-extension of a free group, it is of cohomological dimension at most 2. Hence, by Proposition \ref{gldim}, $K[G]$ is of global dimension at most 2. Therefore, $K[G]$ is coherent by Corollary \ref{crit}.
\end{proof}

Proposition \ref{fp} implies also the following criterion for a group to be of type $\fp_2(k)$.

\begin{cor}
\label{finite}
Let $G$ be a finitely generated group of cohomological dimension two, let $k$ be a commutative ring and suppose that $k[G]$ can be embedded in a division ring $\D$. Then $G$ is of type $\fp_2(k)$ if and only if $\dim_{\D}\Tor_2^{k[G]}(\D, k)$ is finite.
\end{cor}

\begin{proof}
  Since $G$ has cohomological dimension at most two, the left $k G$-module $I_{k[ G]}$ has projective dimension at most one. Let $0\to P_1\to P_0\to I_{k[G]}\to 0$ be a projective resolution. By definition of $\Tor$, we have the exact sequence
\[
\Tor^{k[ G]}_1(\D, I_{k[ G]}) \to \D\otimes_{k[ G]} P_1\to \D\otimes_{k[ G]}P_0.
\]
Since $P_0$ is finitely generated, we see that $P_1$ is finitely generated if and only if
\[
\dim_{\D}\Tor_1^{k[G]}(\D, I_{k[ G]}) = \dim_{\D}\Tor_2^{k[G]}(\D, k) 
\]
is finite by \cref{fp}. This completes the proof.
\end{proof}

Corollary \ref{finite} coupled with \cref{l2betti} yields a proof of \cref{npi}(1) below.

\begin{teo}
If $X$ is a two-complex with non-positive immersions, then $\pi_1(X)$ is homologically coherent.
\end{teo}
\begin{proof} 
By \cref{npi_indicable}, $\pi_1(X)$ is locally indicable. Let $H$ be a finitely generated subgroup of $G$. By \cite{JL20}, $\Z[H]$ is embedded in $\D_{\Q[H]}$. By 
Proposition \ref{l2betti},  
$$b_2^{(2)}(H)=\dim_{\D_{\Q[H]}}\Tor_2^{\Q[H]}(\D_{\Q[H]}, \Q)$$ is finite. Hence, by Corollary \ref{finite}, $H$ is of type $\fp_2(\Z)$.
\end{proof}

The following proposition was communicated to us by Sam Fisher.

\begin{pro}\label{vanishing2nd}
Let $G$ be a locally indicable group of cohomological dimension $n$. Let $K$ be a field and assume that $\D_{K[G]}$ exists. If  $\Tor_n^{K[G]}(\D_{K[G]},K)=0$, then   for any subgroup $H$ of $G$,  $\Tor_n^{K[H]}(\D_{K[H]},K)=0$.
\end{pro}
\begin{proof}
 Observe that, by Shapiro's lemma,
$$   \Tor_n^{K[H]}(\D_{K[H]}, K)
\cong  \Tor_n^{K[G]}(\D_{K[H]}\otimes_{K[H]} K[G], K).$$
By  Theorem  \ref{grater}, the right $K[G]$-module $M=\D_{K[H]}\otimes_{K[H]} K[G]$ can be seen as a submodule of $\D_{K[G]}$. Consider the exact sequence
$$0\to M\to \D_{K[G]}\to  \D_{K[G]}/M\to 0,$$
which induces the exact sequence
$$\Tor_{n+1}^{K[G]}(\D_{K[G]}/M, K)\to \Tor_n^{K[G]}(M, K)\to \Tor_n^{K[G]}(\D_{K[G]}, K).$$
Since  $\Tor_n^{K[G]}(\D_{K[G]},K)=0$ and, by Proposition \ref{gldim}, $K[G]$ is of global dimension at most $n$, we conclude that $\Tor_n^{K[G]}(M, K)=0$. Therefore,  $\Tor_n^{K[H]}(\D_{K[H]},K)=0$.\end{proof}

\begin{cor}\label{vanishing2ndL2}
Let $G$ be a locally indicable group of cohomological dimension $n$ with $b_n^{(2)}(G)=0$. Then for any subgroup $H$ of $G$, $b_n^{(2)}(H)=0$.
\end{cor}
The conclusion of the corollary remains valid without the assumption that $G$ is locally indicable. Although in this paper we only use this result for locally indicable groups, we have decided to sketch the proof of the general case for the convenience of the reader.

\begin{teo}\label{vanishing}
Let $G$ be a group of cohomological dimension $n$ over $\Q$ with $b_n^{(2)}(G) = 0$. Then, for any subgroup $H$ of $G$, we have $b_n^{(2)}(H) = 0$.
\end{teo}

\begin{proof}
Let $\mathcal{U}(G)$ be the ring of affiliated operators of $G$ (see \cite[Section 3]{Ka19}). The $L^2$-Betti number $b_i^{(2)}(G)$ is defined as $\dim_{\mathcal{U}(G)} \Tor_i^{\CC[G]}(\mathcal{U}(G), \CC)$. Since $G$ is of cohomological dimension $n$ over $\Q$, $ \Tor_n^{\CC[G]}(\mathcal{U}(G), \CC)$ is a submodule of a free $\mathcal U(G)$-module. Thus, $\dim_{\mathcal{U}(G)} \Tor_n^{\CC[G]}(\mathcal{U}(G), \CC) = 0$ if and only  $\Tor_n^{\CC[G]}(\mathcal{U}(G), \CC) = 0$. 

The multiplicative map 
$
\mathcal{U}(H) \otimes_{\CC[H]} \CC[G] \to \mathcal{U}(G)
$
is injective. Thus, as in the proof of Proposition \ref{vanishing2nd}, we obtain that 
\[
\Tor_n^{\CC[H]}(\mathcal{U}(H), \CC) \cong \Tor_n^{\CC[G]}(\mathcal{U}(H) \otimes_{\CC[H]} \CC[G], \CC) = \{0\}.
\]
This implies that $b_n^{(2)}(H) = 0$.
\end{proof}

A variation of the previous result can also be found in \cite{GN21}.

Combining Corollary \ref{vanishing2ndL2}  with Corollary \ref{finite},   we obtain  Theorem \ref{npi}(2).

\begin{teo}
If $G$ is a locally indicable group of cohomological dimension two and with $b_2^{(2)}(G) = 0$, then $G$ is homologically coherent.
\end{teo}

Since the only property used in \cref{vanishing2ndL2} was that \( k[G] \) was embeddable into a Linnell division ring, our argument can, in fact, be applied to groups of cohomological dimension two satisfying the strong Atiyah conjecture and having a trivial second \( L^2 \)-Betti number. However, it turns out that these groups are locally indicable.

\begin{teo}
\label{Atiyah_locally}
Let \( G \) be a torsion-free group of cohomological dimension two over \( \mathbb{Q} \) satisfying the strong Atiyah conjecture. If \( b_2^{(2)}(G) = 0 \), then \( G \) is locally indicable.
\end{teo}

\begin{proof}
Let \( H \) be a finitely generated non-trivial subgroup of \( G \). Consider the exact sequence 
\[
0 \to P_1 \to P_2 \to I_{\mathbb{Q}[H]}\to 0,
\]
where \( P_2 \) is a finitely generated projective \( \mathbb{Q}[H] \)-module. Since \( H \) is of cohomological dimension two over \( \mathbb{Q} \), \( P_1 \) is also projective. 

Applying the proof of Corollary \ref{vanishing2ndL2} or directly Theorem \ref{vanishing}, we obtain that \( b_2^{(2)}(H) = 0 \). Hence, by Proposition  \ref{fp}, \( P_1 \) is finitely generated. By \cite[Theorem 8]{Ec01}, for \( i = 1,2 \),
\[
\dim_{\mathbb{Q}}\left(\mathbb{Q} \otimes_{\mathbb{Q}[H]} P_i\right) = \dim_{\mathcal{D}_{\mathbb{Q}[H]}}\left(\mathcal{D}_{\mathbb{Q}[H]} \otimes_{\mathbb{Q}[H]} P_i\right).
\]
Therefore,
\[
b_1(H) = 1 + b_2(H) + b_1^{(2)}(H) - b_2^{(2)}(H) = 1 + b_2(H) + b_1^{(2)}(H) > 0.
\]
Hence \( H \) is indicable, and so \( G \) is locally indicable.
\end{proof}

\section{From homological coherence to coherence}
\label{sect:homcoh}
The reader is referred to Serre's monograph \cite{serre_80} for the necessary background in Bass--Serre theory. We will always assume that our trees and group actions are simplicial. If $G$ is a group acting on a tree $\mathcal{T}$, then we call $\mathcal{T}$ a \textbf{$G$-tree}. If $H\subseteq G$, we denote by $\mathcal{T}/H$ the space obtained from $\mathcal{T}$ by identifying each point $t\in \mathcal{T}$ with $gt$ for each $g\in H$. A $G$-tree $\mathcal{T}$ is \textbf{cocompact} if $\mathcal{T}/G$ is compact. A map between trees $\mathcal{S}\to \mathcal{T}$ is a \textbf{morphism} if it sends vertices to vertices and edges to (possibly trivial) edge paths. A $G$-tree $\mathcal{T}$ is \textbf{dominated} by another $G$-tree $\mathcal{S}$ if there is a $G$-equivariant morphism $\mathcal{S}\to \mathcal{T}$. We call such a morphism a \textbf{domination map}. A \textbf{refinement} of $G$-trees $\mathcal{S}\to \mathcal{T}$ is a domination map that is obtained by collapsing certain edges to vertices.

\begin{lem}
\label{quotient_tree}
Let $G$ be a group and let $\mathcal{S}$ be a $G$-tree. If $N_1, \dots, N_k\le  G$ are subgroups that stabilise vertices of $\mathcal{S}$, then $\mathcal{T} = \mathcal{S}/\normal{N_1, \dots, N_k}$ is a $G/\normal{N_1, \dots, N_k}$-tree.
\end{lem}

\begin{proof}
Let  $N = \normal{N_1, \dots, N_k}$ and consider its action on $\mathcal{S}$. The quotient of $N$ by the normal closure of its elliptic elements $K\trianglelefteqslant N$ is the fundamental group of $\mathcal{S}/N = \mathcal{T}$ (see \cite[Section 5]{serre_80}). Since $N$ is generated by elliptic elements, it follows that $K = N$ and so $\mathcal{T}$ is a $G/N$-tree.
\end{proof}

The following result is \cite[Proposition 2.17]{guirardel_17} due to Guirardel--Levitt, which generalises results of Dunwoody \cite{dunwoody_85}, Baumslag--Shalen \cite[Theorem 1]{baumslag_90} and Fujiwara--Papasoglu \cite[Proposition 5.12]{fujiwara_06}.

\begin{pro}
\label{finitely_presented}
Let $G$ be a group and let 
\[
\mathcal{T}_1\gets\ldots \gets\mathcal{T}_k\gets \mathcal{T}_{k+1}\gets\ldots
\]
be a sequence of refinements of cocompact $G$-trees. If $G$ is finitely presented, then there is a cocompact $G$-tree $\mathcal{S}$ with finitely generated vertex and edge stabilisers that dominates $\mathcal{T}_i$ for all $i$.
\end{pro}

Before proceeding, we first need the following useful lemma due to Bieri--Strebel \cite[Lemma 2.1]{bieri_78}. If $G$ is a group, then $G_{\text{ab}}$ denotes its abelianisation.

\begin{lem}
\label{bieri_lemma}
Let $k $ be a commutative ring with $1\neq 0$ and let $G$ be a group. Then $G$ is of type $\fp_2(k )$ if and only if there is a short exact sequence
\[\begin{tikzcd}
1 \arrow[r] & N \arrow[r] & H \arrow[r] & G \arrow[r] & 1
\end{tikzcd}\]
where $H$ is finitely presented and $N_{\text{ab}}\otimes_{\Z}k =0$.
\end{lem}

We now generalise \cref{finitely_presented} to groups of type $\fp_2(k )$ for an arbitrary commutative ring $k $ with $1$. The proof is essentially an extension of the proof of \cite[Theorem A]{bieri_78}. When $k  = \Z/2\Z$ and  $\mathcal{T}_i = \mathcal{T}$ for all $i$, the following is due to Dicks--Dunwoody \cite[Theorem 4.4]{dicks_89}.

\begin{teo}
\label{fp2_action}
Let $k $ be a commutative ring with $1\neq 0$, let $G$ be a group and let 
\[
\mathcal{T}_1\gets \ldots \gets\mathcal{T}_k\gets \mathcal{T}_{k+1}\gets\ldots
\]
be a sequence of refinements of cocompact $G$-trees. If $G$ is of type $\fp_2(k )$, then there is a cocompact $G$-tree $\mathcal{S}$ with finitely generated vertex and edge stabilisers that dominates $\mathcal{T}_i$ for all $i$.
\end{teo}

\begin{proof}
Since $G$ is of type $\fp_2(k )$, there is a short exact sequence
\[
\begin{tikzcd}
1 \arrow[r] & N \arrow[r] & H \arrow[r, "\pi"] & G \arrow[r] & 1
\end{tikzcd}
\]
where $H$ is finitely presented and $N_{\text{ab}}\otimes_{\Z}k =0$ by \cref{bieri_lemma}. Now $\pi$ turns $\mathcal{T}_i$ into an $H$-tree. Since $H$ is finitely presented, by \cref{finitely_presented} there is a cocompact $H$-tree $\mathcal{S}'$ with finitely generated vertex and edge stabilisers that dominates the $H$-trees $\mathcal{T}_i$ for all $i$. Denote by $v_1, \ldots, v_n$ orbit representatives of vertices in $\mathcal{S}'$. Denote by $H_1, \ldots, H_n \le  H$ the corresponding stabiliser subgroups, which are finitely generated. Let $N_i = \ker\left(\pi|_{H_i}\right)$. We now consider the quotients 
\begin{align*}
Q &= H/\normal{N_1, \ldots, N_n},\\
\mathcal{S} &= \mathcal{S}'/\normal{N_1, \ldots, N_n}. 
\end{align*}
By \cref{quotient_tree}, $\mathcal{S}$ is a $Q$-tree. Since $\mathcal{S}/Q\isom \mathcal{S}'/H$, we see that $\mathcal{S}$ is a cocompact $Q$-tree. As $\pi$ factors through $H\to Q$, we obtain domination maps of cocompact $Q$-trees $\mathcal{S}\to \mathcal{T}_i$. Denote by $N_Q$ the image of $N$ in $Q$. By definition, $N_Q$ intersects each vertex stabiliser of the $Q$-tree $\mathcal{S}$ trivially and thus acts freely on $\mathcal{S}$. A group acting freely on a tree is free and so $N_Q$ must be free. Since $N_{\text{ab}}\otimes_{\Z}k =0$ and $N_Q$ is a quotient of $N$, this implies that $(N_Q)_{\text{ab}}\otimes_{\Z}k =0$. If $F$ is a free group, then $F_{\text{ab}}\otimes_{\Z}k  = 0$ if and only if $F$ is trivial. Thus, $N_Q = 1$ and $Q\isom G$. This makes $\mathcal{S}$ a cocompact $G$-tree, dominating the $G$-trees $\mathcal{T}_i$. The vertex stabilisers of the $G$-tree $\mathcal{S}$ are conjugates of $H_1/N_1, \ldots, H_n/N_n$ and so are finitely generated. Similarly, the edge stabilisers are conjugates of quotients of the edge stabilisers of the $H$-tree $\mathcal{S}'$ and so are also finitely generated. This completes the proof.
\end{proof}

Now we prove Theorem \ref{equivalence}. 

\begin{teo} \label{equivalence1} Let $k$ be a commutative ring with $1\neq 0$, let $G$ be a group acting on a tree $\mathcal{T}$ with coherent vertex stabilisers. If $H\le  G$ is of type $\fp_2(k) $, then $H$ is finitely presented. In particular, if $G$ is homologically coherent, then $G$ is coherent.
\end{teo}
 
\begin{proof}
Let $H\le  G$ be a finitely generated subgroup. Since $H$ is finitely generated, there is an $H$-invariant subtree $\mathcal{S}\subset \mathcal{T}$ such that $\mathcal{S}$ is a cocompact $H$-tree. Now each vertex stabiliser for the action of $H$ on $\mathcal{S}$ is contained in a vertex stabiliser for the action of $G$ on $\mathcal{T}$. Since $H$ has type $\fp_2(k )$, by \cref{fp2_action}, $H$ splits as a finite graph of groups with finitely generated vertex and edge groups, each conjugate into vertex stabilisers for the action of $G$ on $\mathcal{T}$. Since the vertex stabilisers for the action of $G$ on $\mathcal{T}$ are coherent, it follows that $H$ splits as a finite graph of groups with finitely presented vertex and edge groups. Hence, $H$ is finitely presented.
\end{proof}

We record the following immediate consequence of \cref{equivalence1} for future use.

\begin{teo}
\label{hierarchy}
Denote by $\mathcal{CG}$ the smallest class of groups containing all coherent groups that is closed under finite extensions, subgroups, amalgamated free products, HNN-extensions and directed unions. Let $k$ be a commutative ring with $1\neq 0$. If $G\in \mathcal{CG}$ has the property that all of its finitely generated subgroups are of type $\fp_2(k)$, then $G$ is coherent.
\end{teo}

We are now ready to settle Baumslag's conjecture.
\begin{teo}
If $G$ is a one-relator group, then 
$G$ is coherent.
\end{teo}
\begin{proof}  One-relator groups have vanishing second $L^2$-Betti number by work of Dicks--Linnell \cite{DL07} and hence, as do all of their finite index subgroups. One-relator groups are virtually torsion-free by Fisher--Karrass--Solitar \cite{FKS72} and hence, by work of Brodskii \cite{Br84} (see also \cite[Corollary 3.2]{Ho00}), they are virtually locally indicable. Since one-relator groups virtually have cohomological dimension at most two, applying \cref{npi}(2) we see that one-relator groups are homologically coherent. Alternatively, $G$ has a finite index subgroup that is the fundamental group of a two-complex with non-positive immersions \cite[Theorem 6.1]{HW16} and so we could also use \cref{npi}(1) to conclude that $G$ is homologically coherent. We could also deduce that every finitely generated subgroup of $G$ is of type $\fp_2(\Q)$ from the coherence of $\Q [G]$ proved in Theorem \ref{th:coherentgroupalgebra}.

If $G$ is a one-relator group, the Magnus--Moldavanskii hierarchy (in the form of \cite{masters_06} or \cite[Theorem 5.2]{linton_22}) tells us that there is a finite sequence of one-relator subgroups
\[
G_N \le \ldots\le G_1\le  G_0 = G
\]
such that $G_N$ is a free product of a free group and a finite cyclic group and $G_i$ splits as a HNN-extension over $G_{i+1}$. Since $G$ is homologically coherent, \cref{hierarchy} implies that $G$ is coherent.
\end{proof}

\section{Construction of flat modules}
\label{sect:flat}
In this section we will show that certain modules are flat. This will be used later in calculations of different Tor functors.

\subsection{Some auxiliary results}
Let $R$ be a ring. A  left  $R$-module $M$ is {\bf torsion-free} if for any $0\ne r\in R$, $m\in M$, we have that $r\cdot m = 0$ implies that  $m = 0$.

\begin{lem}\label{injectiveaction} Let $K$ be a field  and let $F$ be a free group.
Let $V$ be a left $K[F]$-module and let $G=F/N$ be a locally indicable group. Let $$0\ne \alpha =\displaystyle \sum_{i=1}^n c_i\cdot f_i\in K [F]\ \left (0\ne c_i\in K ,\ f_i\in F\right ) .$$
Assume that all $g_i=f_iN\in G$ are different.
If $\D_{K[G]}$ exists, then $$\Ann_{V\otimes_K  \D_{K [G]}}\left (\alpha\right ) =\{m\in V\otimes_K \D_{K [G]}\colon \alpha \cdot m=0\}$$ is trivial. In particular, if $V$ is a left $K[G]$-module, then $V\otimes_K \D_{K[G]}$ is torsion-free as a left $K[G]$-module.
\end{lem} 
\begin{proof}
We prove the lemma by induction on $n$. If $n=1$, the statement is clear. 

Consider now the case where $n>1$ and assume that $ \Ann_{V\otimes_K  \D_{K [G]}}\left (\alpha\right ) \ne 0$. Without loss of generality we can assume that $f_1=1$. Let $H=\langle g_2,\ldots, g_n\rangle$. Since $n>1$ and all $g_i$ are different, $H$ is not trivial. Let $\widetilde H=\langle f_2,\ldots, f_n\rangle$. Since $G$ is locally indicable, there exists an epimorphism $ \phi: H\to \Z$, which induces an epimorphism $\widetilde \phi: \widetilde H\to \Z$ satisfying $\widetilde \phi\left (x\right ) =\phi\left (xN\right ) $.

Let $s\in \widetilde H$ be such that $\widetilde \phi \left (s\right ) =1$. Then we can write
$$\alpha=\sum_{j=a}^b \alpha_j\cdot s^j, \textrm{\ with\ } \alpha_j\in R [\ker \widetilde \phi] \textrm{\ and\ }\alpha_b\ne 0.$$
Observe that if we write $$\alpha_b=\sum_{k=1}^l d_k\cdot h_k\ \left (0\ne d_k\in R  , \ h_k\in F \right ) ,$$ then all $h_kN$ are different and $l<n$.

For simplicity we will write  $\D$ instead of $\D_{K [G]}$. We can see $V\otimes_K   \D$ as an $\left (K[F], \D\right ) $-bimodule. Therefore, given a basis $B$ of $\D$ as  a left  $\D_H$-module we obtain that $$V\otimes_K  \D=\oplus_{q\in B} \left (V\otimes_K  \D_Hq\right ) =\oplus_{q\in B} \left (V\otimes_K  \D_H\right ) q.$$
Thus, since $ \Ann_{V\otimes_K  \D}\left (\alpha\right ) \ne 0$, $\Ann_{V\otimes_K   \D_{H}}\left (\alpha\right )  \ne 0$ as well. 

Let $t=sN\in H$. Since $\D$ is Hughes-free, the ring  $S$  generated by $\D_{\ker \phi}$ and $t$ is isomorphic to the ring of twisted polynomials $\D_{\ker \phi}[t^{\pm1}, \tau]$, where $\tau$ is the automorphism of $\D_{\ker \phi}$ induced by conjugation by $t$. Observe that $\D_H$ is the Ore ring of fractions of $S$. Thus, since $\Ann_{V\otimes_K   \D_{H}}\left (\alpha\right )  \ne 0$, we also obtain that $\Ann_{V\otimes_K  S}\left (\alpha\right )  \ne 0$. 

Let $0\ne m\in \Ann_{V\otimes_K  S}\left (\alpha\right ) $. We can write $$m=\sum_{j=c}^d m_j \cdot t^j, \textrm{\ with\ } m_j\in V\otimes_K   \D_{\ker \phi}\textrm{\ and\ }m_d\ne 0.$$
Observe that for every $j$, $$s^j \left (V\otimes_K  \D_{\ker \phi}\right ) =V\otimes_K   t^j\D_{\ker \phi}=\left ( V\otimes_K   \D_{\ker \phi}\right ) t^j.$$
Thus, since,  $\alpha\cdot m=0$, $\alpha_bs^b\cdot m_dt^d=0$. Since $m_d\ne 0$, $s^b\cdot m_dt^d\ne 0$. Hence $\Ann_{V\otimes_K   \D}\left (\alpha_b\right ) \ne 0$. But this contradicts the inductive hypothesis.
\end{proof}

Given a left   $K [G]$-module $M$, let $M^*$ be the right   $K [G]$-module that coincides with $M$ as a $K $-vector space and the action of $G$ is given by
$m\cdot g=g^{-1} m$. In the same way, given a right $K [G]$ module $M$ we can define the left $K [G]$-module $M^*$.  
\begin{lem}\label{dual}
Let $G$ be a group. The following properties hold.
\begin{enumerate}
\item Let $W\le G$. Then we have that $\left (I_{K [G]}/\left ({}^GI_{K[W]}\right ) \right ) ^*\cong I_{K [G]}/I_{K[W]}^G$.
\item Given left $K [G]$-modules $M$ and $N$ and a right $K [G]$-module $L$, we have that for every $k\in \N$,
$$\Tor_k^{K [G]}\left (L,M\otimes_K  N\right ) \cong \Tor_k^{K [G]}\left (N^*,M\otimes_K  L^*\right ) .$$
\end{enumerate}
\end{lem}
\begin{proof}
(1)  is clear   and the proof. 

To see (2), \cite[Proposition III.2.2]{Br82} yields the following
\begin{align*}
\Tor_k^{K[G]}(L, M\otimes_KN) &\cong \Tor_k^{K[G]}(K, L^*\otimes_K(M\otimes_KN))\\
\Tor_k^{K[G]}(N^*, M\otimes_KL^*) &\cong \Tor_k^{K[G]}(K, N\otimes_K(M\otimes_KL^*)).
\end{align*}
Since $L^*\otimes_K(M\otimes_KN) \cong N\otimes_K(M\otimes_KL^*)$, the proof is complete.
\end{proof}

\subsection{Proof of Theorems \ref{weak}, \ref{th:coherentgroupalgebra} and \ref{main}(2)}
In this subsection we construct a specific flat $K[G]$-module where $G$ is the fundamental group of a reducible two-complex without proper powers.  This  result  implies Theorem \ref{weak}.
\begin{teo}\label{key}
 Let $K$ be a field and  $G=\pi_1(X)$, where $X$ is a finite reducible two-complex without proper powers. Assume that $\D_{K[G]}$ exists. Then  the left $K [G]$-module $\D_{K [G] }\otimes_K  I_{K[G]}$ is flat.
\end{teo}

 \begin{proof}
We prove the theorem by induction on the number of two-cells in $X$.  If there are no two-cells, then the module $I_{K[G]}$ is free, and so  $\D_{K[G]}\otimes_K I_{K[G]}$ is free as well. The division ring $\D_{K[G]}$ in this case always exists for any field $K$ by a result of Lewin \cite{Le74}.

Since $X$ is  a finite reducible two-complex without proper powers,  $G=G_1*G_2/\normal{w}$, where $G_1$ and $G_2$
 are the fundamental groups of reducible two-complexes   and  $w$ is either 1 or  $w$ is   not conjugate into $G_1$ or $G_2$ within $G_1*G_2$ and is not equal to a proper power in $G_1*G_2$.

We have that by the inductive hypothesis $\D_{K [G_i] }\otimes_K  I_{K[G_i]}$ is flat as a left $K[G_i]$-module. Observe also that by \cite[Theorem 4.3]{howie_81} we can view $G_i$ ($i=1,2$)  as subgroups of $G$.
\begin{claim} \label{flatness}
For $i=1,2$,
$\D_{K [G] }\otimes_K \left  ({}^GI_{K [G_i]} \right  ) $ is flat as a left $K[G]$-module.
\end{claim}

\begin{proof} Observe that since the division subalgebra of $\D_{K [G] }$ generated by $K[G_i]$ is isomorphic (as a $K[G_i]$-ring)  to $\D_{K [G_i] }$, $\D_{K [G] }\otimes_K  I_{K[G_i]}$ is also flat as a left $K[G_i]$-module.

Let $L$ be a right $K[G]$-module. Then by Lemma \ref{dual}, Lemma \ref{augmentation} and Shapiro's lemma,
\begin{multline*}
\Tor_1^{K[G]}\left (L,\D_{K [G] }\otimes_K \left  ({}^GI_{K [G_i]} \right  ) \right )\cong \Tor_1^{K[G]}\left ( I_{K [G_i]}^G 
,  \D_{K [G] }\otimes_K L^*  \right ) \cong\\  \Tor_1^{K[G_i]} \left (I_{K[G_i]}, \D_{K [G] }\otimes_K L^*  \right ) \cong \Tor_1^{K[G_i]} \left (L,  \D_{K [G] }\otimes_K I_{K[G_i]} \right ) =\{0\}.
\end{multline*}
Therefore, $\D_{K [G] }\otimes_K \left  ({}^GI_{K [G_i]} \right  ) $ is flat.
\end{proof}
From the previous claim the theorem follows in the case $w=1$. Thus, from now on  we assume that $w\ne 1$.
 
We can write
$$ w-1=\alpha_1+\alpha_2  \textrm {\ with\ }\alpha_1\in I_{K[G_1]} ^{G_1*G_2} \textrm {\ and\ }\alpha_2\in {}I_{K[G_2]}^{G_1*G_2}.$$
 
\begin{claim}   \label{nontrivial}
For $i=1,2$ the image $\overline{\alpha_i}$ of $\alpha_i$ in $K[G]$ is not trivial. 
\end{claim}

\begin{proof}
Replacing $w$ by a conjugate if necessary, we can write $w = a_1b_1\ldots a_nb_n$, with $n\ge 1$, $1\ne a_i\in G_1$ and $1\ne b_i\in G_2$. We have that
 $$\alpha_1=\sum_{i=1}^n \left (a_i-1\right ) b_i a_{i+1}b_{i+1}\ldots a_nb_n \textrm{\ and \ }  \alpha_2=\sum_{i=1}^n \left (b_i-1\right )  a_{i+1}b_{i+1}\ldots a_nb_n.$$
By \cite[Corollary 3.4]{Ho82}, we see that there are no distinct suffixes $u, v$ of $ w$ (as a word over $G_1$ and $G_2$)  such that $u = v$ in $G$. This implies the claim.
\end{proof} 
We will write $\D$ instead of $\D_{K[G]}$ for ease of notation. Let $L$ be a right $K [G]$-module. Then by Lemma \ref{dual},
$$\Tor_1^{K [G]}\left (L,\D\otimes_{K} I_{K[G]}\right ) \cong \Tor_1^{K [G]}\left ( I_{K[G]} ,\D\otimes_K  L^*\right ) .$$

Consider the exact sequence
\[
0\to N_i\to F_i\to I_{K[G_i]}\to 0
\]
where $N_i$ is a relation module for $G_i$ and where $F_i$ is a free right $K[G_i]$-module for $i = 1, 2$. Howie showed in \cite[Theorem 10]{Ho84} that we have the following exact sequence
\begin{multline*}
0\to K[G]\oplus\left(N_1\otimes_{K[G_1]}K[G]\right)\oplus\left(N_2\otimes_{K[G_1]}K[G]\right) \to \\
(F_1\otimes_{K[G_1]}K[G])\oplus (F_2\otimes_{K[G_2]}K[G]) \to I_{K[G]}\to0.
\end{multline*}
Quotienting $F_i\otimes_{K[G_i]}K[G]$ by the image of $N_i\otimes_{K[G_i]}K[G]$ for $i = 1, 2$ leads us to the following exact sequence of right \( K[G] \)-modules:

 $$0\to K[G]\xrightarrow{\gamma} I_{K[G_1]}^G \oplus  I_{K[G_2]}^G\to  I_{K[G]} \to 0,$$
 where $\gamma\left (a\right ) = \left (\overline{\alpha_1} a, \overline{\alpha_2} a\right ) $. Thus, we obtain  the exact sequence
\begin{multline*}
\Tor_1^{K[G]}\left (I_{K[G_1]}^G \oplus  I_{K[G_2]}^G , \D\otimes_K  L^*\right ) \to \Tor_1^{K[G]}\left (I_{K[G]}  ,\D\otimes_K  L^*\right ) 
\to\\ \D\otimes_K  L^*\xrightarrow{\widetilde \gamma}\left (\left (I_{K[G_1]}^G\right ) \otimes_{K[G]}\D\otimes_K  L^*\right )  \oplus  \left (\left (I_{K[G_2]}^G\right ) \otimes_{K[G]}\D\otimes_K  L^*\right ) ,  \end{multline*}
  where $\widetilde \gamma\left (m\right ) = \left ( \overline{\alpha_1}\otimes m,\overline{\alpha_2}\otimes m\right ) $. Observe that the composition
  $$m\mapsto \overline{\alpha_1}\otimes m \mapsto \overline{\alpha_1} m$$ is the multiplication by $\overline{\alpha_1}$. By Claim \ref{nontrivial}, $\overline{\alpha_1}\ne 0$ and by Lemma  \ref{injectiveaction}, $ \D\otimes_K  L^*$ is torsion-free. Thus, $\ker \widetilde \gamma =\{0\}$. On the other hand by Claim \ref{flatness}, $$\Tor_1^{K[G]}\left (I_{K[G_1]}^G \oplus  I_{K[G_2]}^G , \D\otimes_K  L^*\right ) $$ is trivial. Thus, $$\Tor_1^{K [G]}\left (L,\D\otimes_{K} I_{K[G]}\right ) \cong \Tor_1^{K [G]}\left ( I_{K[G]} ,\D\otimes_K  L^*\right ) =\{0\},$$ and so  
  $\D\otimes_K  I_{K[G]} $ is flat.
\end{proof}
Theorem \ref{weak} now follows directly from the previous theorem.
\begin{proof}[Proof of Theorem \ref{weak}]
By symmetry we also have that the right $K[G]$-module $$\D_{K[G]}\otimes_K I_{K[G]}$$ is flat. Therefore, using Lemma \ref{dual},   we obtain that for any left $K[G]$-module $L$,
\begin{multline*}
\Tor_2^{K[G]}\left (\D_{K[G]}, L\right ) \cong \Tor_2^{K[G]}\left (\D_{K[G]} \otimes _K L^*, K\right ) \cong \\ \Tor_1^{K[G]}\left (\D_{K[G]}\otimes _K L^*, I_{K[G]}\right ) \cong  \Tor_1^{K[G]}\left (\D_{K[G]} \otimes _K   I_{K[G]}, L\right ) =\{0\}.\end{multline*}
Hence,  the right $K[G]$-module  $\D_{K[G]}$ is of weak dimension at most $1$.
\end{proof}
\begin{proof}[Proof of Theorem \ref{th:coherentgroupalgebra}]

By \cref{reducible_npi} and \cref{npi_aspherical}, we see that $X$ is an aspherical two-complex and $\pi_1(X)$ is locally indicable (see also \cite{Ho82}). Thus, $\pi_1(X)$ has cohomological dimension at most two, and so,  by Proposition \ref{gldim},  $K[G]$ has global dimension at most two. Since $\pi_1(X)$ is locally indicable, we see that $\D_{K[G]}$ exists by \cite{JL20}. By \cref{weak}, $\D_{K[G]}$ has weak dimension at most one. Now \cref{crit} tells us that $K[G]$ is coherent.
\end{proof}

\begin{proof}[Proof of Theorem \ref{main}(2)]

In  the case that $G$ is torsion-free the result follows from Theorem \ref{th:coherentgroupalgebra}.  If  $G$ has torsion,  by work of Kielak and the second author \cite{KL23}, $G$ is virtually free-by-cyclic. The fact that $K[G]$ is coherent follows  from \cite[Proposition 2.9]{HL22} or,  alternatively, we could use \cref{mappingtorus}.

\end{proof}

\subsection{Flat modules associated with Magnus subgroups}
\label{sub:Magnus}
Let $X = (\Gamma, \lambda)$ be a finite bireducible two-complex. Denote by $\overline X=(\Gamma, \overline \lambda)$ the bireducible two-complex such that  $\overline \lambda\colon\mathbb{S}\immerses \Gamma$ is primitive and there exists a cover $\mu\colon \mathbb{S} \immerses  \mathbb{S}$ with $\lambda =\overline  \lambda \circ \mu$. By Lemma \ref{proper_powers},  $\overline X$ is without proper powers. As we have already mentioned, Howie proved in \cite{Ho82} that $\overline G=\pi_1(\overline X)$ is locally indicable. The group $\overline G$ is also the maximal torsion-free quotient of $G$.

\begin{teo}\label{key3}
Let $G=\pi_1(X)$, where $X$ is a finite bireducible two-complex,   $\overline G=\pi_1(\overline X)$ and  $H=\pi_1(\Lambda)$ a Magnus subgroup of  $G$ where $\Lambda\subset X$ is small. Let $K$ be a field of characteristic coprime with the orders of finite elements of $G$  and  assume that $\D_{K[\overline G]}$ exists. Then the left $K[G]$-module 
$$\D_{K[\overline G]}\otimes \left (I_{K[G]}/\left ( {}^G I_{K[H]}\right )\right )$$ is flat.
\end{teo}
\begin{proof}
We prove the theorem by induction on the number of 2-cells of $X$. 
In the case where there are no 2-cells, \( G \) is a free group, and \( H \) is a free factor of \( G \). Thus, \( I_{K[G]} / \big({}^G I_{K[H]}\big) \) is free, and so the result follows.

Assume now that $X$ has   at least one 2-cell. 

Let $Z\subset X$ be a reduction such that $\Lambda \subset Z$ and let $\alpha$ be the two-cell in $X-Z$. We can attach an edge $e$ to $X$ so that $Z^\prime =Z\cup e$ is connected. Put $X^\prime =X\cup e$. Then $G^\prime=\pi_1(X^\prime)\cong G*\Z$. 

\begin{claim}\label{connect}
If the theorem holds for the pair $H\le G^\prime$, then it   holds also for the pair $H\le G$.
\end{claim}
\begin{proof}
 Let \(\overline{X'}\) be the bireducible two-complex without proper powers associated with \(X'\), and let \(\overline{G'} = \pi_1(\overline{X'})\).
We assume that the left $K[G^\prime]$-module $M=\D_{K[\overline {G^\prime}]}\otimes \left (I_{K[G^{\prime}]}/\left ( {}^G I_{K[H]}\right )\right )$ is flat.
 Let $L$ be any right $K[G]$-module. Then by Lemma \ref{dual}, Lemma \ref{augmentation} and Shapiro's lemma,
\begin{multline*}
\Tor_1^{K[G]}\left (L, \D_{K [\overline {G^\prime}] }\otimes_K\left  ({}I_{K [G]}/{}^GI_{K [H]} \right  ) \right )\cong\\  
\Tor_1^{K[G]}\left ( I_{K [G]}/I_{K [H]}^G
,  \D_{K [\overline {G^\prime}] }\otimes_K L^*  \right ) \cong
 \Tor_1^{K[G^\prime]} \left (I_{K [G]}^{G^{\prime}}/ I_{K [H]}^{G^\prime}, \D_{K [\overline {G^\prime}] }\otimes_K L^*  \right ) \cong\\ \Tor_1^{K[G^\prime]} \left (L,  \D_{K [\overline {G^\prime}] }\otimes_K  \left ({}^{G^\prime}I_{K [G]}/\left ({}^{G^\prime} I_{K [H]}\right) \right ) \right ) .
\end{multline*}
Since $G^\prime=G*\Z$,  $I_{K [G^\prime ]}\cong \left({}^{G^\prime}I_{K [G]}\right )\oplus K [G]$ as left $K[G]$-modules,
and so,  $K[G^\prime]$-module $\left({}^{G^\prime}I_{K [G]}\right )/\left ({}^{G^\prime} I_{K [H]}\right)$ is a direct summand of $ I_{K [G^\prime ]}/\left ({}^{G^\prime} I_{K [H]}\right)$. Hence, 
$\D_{K [\overline {G^\prime}] }\otimes_K  \left ({}^{G^\prime}I_{K [G]}/\left ({}^{G^\prime} I_{K [H]}\right) \right )$ is a direct summand of $M$, which is flat. We conclude that
$\Tor_1^{K[G]}\left (L, \D_{K [\overline {G^\prime}] }\otimes_K\left  ({}I_{K [G]}/{}^GI_{K [H]} \right  ) \right )=\{0\}.$
Therefore the left $K[G]$-module $\D_{K [\overline {G^\prime}] }\otimes_K\left  ({}I_{K [G]}/{}^GI_{K [H]} \right  )$ is flat. On the other hand,  the left $K[G]$-module $\D_{K[\overline G]} $ is a direct summand of 
$\D_{K[\overline {G^\prime}]}$. Thus,  $\D_{K[\overline G]}\otimes \left (I_{K[G]}/\left ( {}^G I_{K[H]}\right )\right )$ is also  flat.
\end{proof}
By Claim \ref{connect}, we can assume that $Z$ is connected.
Hence, we are in the following situation:
\begin{enumerate}
\item $A=\pi_1(Z)$.
\item $G=A*\langle t\rangle/\normal{w}$, where $w=u^l$ is a word over the free product $A*\langle t\rangle$, representing the boundary of $\alpha$ in $\pi_1(X)$.
\item $H$ is a Magnus subgroup of $A$.
\end{enumerate}

By the induction hypothesis we have that $\D_{K [\overline A] }\otimes_K \left (I_{K[A]}/\left ( {}^A I_{K[H]}\right )\right )$ is flat   as a left $K[A]$-module. Observe also that by \cite[Theorem 4.3]{howie_81} we can view $\overline A$  as a subgroup of $\overline G$.
\begin{claim} \label{flatness3}
We have that 
$\D_{K [\overline G] }\otimes_K \left  ({}^GI_{K [A]}/{}^GI_{K [H]} \right  ) $ is flat as a left $K[G]$-module.
\end{claim}

\begin{proof} Observe that since the division subalgebra of $\D_{K [\overline G] }$ generated by $K[\overline A]$ is isomorphic (as a $K[\overline A]$-ring)  to $\D_{K [\overline A] }$, $\D_{K [\overline G] }\otimes_K \left (I_{K[A]}/\left ( {}^A I_{K[H]}\right )\right )$   is also flat as a left $K[A]$-module.

Let $L$ be a right $K[G]$-module. Then by Lemma \ref{dual}, Lemma \ref{augmentation} and Shapiro's lemma,
\begin{multline*}
\Tor_1^{K[G]}\left (L, \D_{K [\overline G] }\otimes_K\left  ({}^GI_{K [A]}/{}^GI_{K [H]} \right  ) \right )\cong\\  \Tor_1^{K[G]}\left ( I_{K [A]}^G/I_{K [H]}^G
,  \D_{K [\overline G] }\otimes_K L^*  \right ) \cong
 \Tor_1^{K[A]} \left (I_{K [A]}/ I_{K [H]}^A, \D_{K [\overline G] }\otimes_K L^*  \right ) \cong\\ \Tor_1^{K[G_i]} \left (L,  \D_{K [\overline G] }\otimes_K  \left (I_{K [A]}/{}^A I_{K [H]} \right ) \right )=0.
\end{multline*}
Therefore, $\D_{K [\overline G] }\otimes_K \left  ({}^GI_{K [A]}/{}^GI_{K [H]} \right  ) $  is flat.
\end{proof}
 
  We can write
$$ u-1=\alpha_1+(t-1) \beta \textrm {\ with\ }\alpha_1\in I_{K[A]} ^{A*\langle t\rangle} \textrm {\ and\ } \beta\in K[A*\langle t\rangle].$$

\begin{claim}   \label{nontrivial3}
If we write $\beta  =\displaystyle \sum_{i=1}^n c_i\cdot f_i\in K [G]\ \left (0\ne c_i\in K ,\ f_i\in G\right )$. Then
  all the images of $f_i$ in $\overline G$ are different. \end{claim}
\begin{proof}
Denote by $X_{\alpha}\subset X$ the smallest subcomplex containing the two-cell $\alpha$. By  Lemma \ref{injective_sub}, $X_\alpha$ is a $\pi_1$-injective  subcomplex of  $X$. Observe also that $\overline{X_{\alpha}}$ is  a $\pi_1$-injective  subcomplex of  $\overline X$.  Thus, in this claim we can assume that $X=X_\alpha$.

Hence $\pi_1(X)=\langle t, x_1,\ldots, x_d|u^l\rangle$, where $u$ is not a proper power in the free group $F$ generated freely by 
$\{t, x_1,\ldots, x_d\}$. Put $x_0=t$
and let $u=x_{i_1}^{\epsilon_1}\ldots x_{i_k}^{\epsilon_k}$ be the reduced  form of $u$ (here $\epsilon_i=\pm 1$). We can also assume that $u$ is cyclically reduced. Since $u$ is not conjugate to a word in $\langle x_1,\ldots, x_d\rangle$, we have that  $x_0$ appears in the expression of $u$. Write
$$u-1= (x_0-1)\alpha_0+\cdots+ (x_d-1)\alpha_d, \textrm{\ with\ }\alpha_i\in K[F].$$
The support of $\alpha_0$ consists of words 
$x_{i_m}^{\epsilon_m}\ldots x_{i_k}^{\epsilon_k}$, where $$1\le m\le k+1 \textrm{\ and\ } x_{i_m}^{\epsilon_m}=x_0^{-1} \textrm{\ or\ }  x_{i_{m-1}}^{\epsilon_{m-1}}=x_0.$$
 Observe that, since $u$ is reduced, both cases cannot occur and
  since $u$ is cyclically reduced, both $1$ and $u$ cannot be in the support of $\alpha_1$. By \cite{We72} the image in $\overline G\cong \langle x_0, x_1,\ldots, x_d|u\rangle$ of every  proper non-trivial subword of $u$ is non-trivial. This implies the claim.
 \end{proof}

 \begin{claim} \label{presentation} We have the following exact sequence of right $K[G]$-modules:
 $$0\to K[G]/(u-1)K[G]\xrightarrow{\gamma} I_{K[A]}^G  \oplus K[G]\to   I_{K[G]} \to 0,$$
  where $\gamma=(\gamma_1,\gamma_2)$ and $\gamma_2(x ) = \beta \cdot (u^{l-1}+\ldots+u+1)  \cdot x  $. 
  \end{claim}
 \begin{proof}
 This is well-known when $X=X_\alpha$, or rather when $G$ is a one-relator group, by Lyndon's identity Theorem (see \cite[Section 11]{Lyn50}). In the general case we have that
   \[
X = Z\cup_{\Lambda}X_{\alpha}
\]
where $\Lambda = Z\cap X_{\alpha}$.  Put $T=\pi_1(\Lambda)$ and $S=\pi_1(X_\alpha)$. This leads to the exact sequences
\begin{multline*}
0\to I_{K[T]}^G\to I_{K[A]}^G\oplus I_{K[S]}^G\to I_{K[G]}\to 0 \textrm{\ and\ } \\ 0\to K[G]/(u-1)K[G]\xrightarrow{\gamma} I_{K[T]}^G\oplus K[G]\to I_{K[S]}^G\to 0,\end{multline*}
 where $\gamma=(\gamma_1,\gamma_2)$ and $\gamma_2(a ) = \beta \cdot (u^{l-1}+\ldots+u+1)  \cdot a$.  Combining these two exact sequences we obtain the claim.
 \end{proof}
 We will write $\D$ instead of $\D_{K[\overline G]}$.
 Let $L$ be a right $K [G]$-module. Then by Lemma \ref{dual},
$$\Tor_1^{K [G]}\left (L,\D \otimes_K \left  ( I_{K[G]}/\left ( {}^G I_{K[H]}\right ) \right  ) \right ) \cong \Tor_1^{K [G]}\left (  I_{K[G]}/\left (  I_{K[H]}^G\right ) ,\D\otimes_K  L^*\right ) .$$
By Claim \ref{presentation}, we have the following exact sequence of right $K[G]$-modules:
 $$ \ker \overline \gamma \to K[G]/(u-1)K[G]\xrightarrow{\overline \gamma} \left ( I_{K[A]}^G/\left (  I_{K[H]}^G\right ) \right )\oplus K[G]\to   I_{K[G]}/\left (  I_{K[H]}^G\right ) \to 0,$$
 where $\overline \gamma=(\overline \gamma_1,\gamma_2)$ and $\gamma_2(a ) = \beta \cdot (u^{l-1}+\ldots+u+1)  \cdot a  $. Since $\gamma_2$ is injective, $\ker \overline \gamma=0$.
  Thus, we obtain  the exact sequence
\begin{multline*}
\Tor_1^{K[G]}\left(  \left ( I_{K[A]}^G/\left ( I_{K[H]}^G\right ) \right )\oplus K[G]\ , \D\otimes_K  L^*\right ) \to \\ \Tor_1^{K[G]}\left (   I_{K[G]}/\left ( I_{K[H]}^G\right )   ,\D\otimes_K  L^*\right ) 
\to\\ 
\D\otimes_K  L^*/(u-1)\left (\D\otimes_K  L^*\right )\xrightarrow{\widetilde \gamma}
\left ( I_{K[A]}^G/\left ( I_{K[H]} ^G  \right )  \otimes_{K[G]} \left ( \D\otimes_K  L^*\right ) \right )  \oplus \left (  \D\otimes_K  L^* \right ),  \end{multline*}
  where $\widetilde \gamma=(\widetilde \gamma_1,\widetilde \gamma_2)$ with $\widetilde \gamma_2(m)=\beta\cdot (u^{l-1}+\ldots+u+1)  \cdot m$.

  By Claim \ref{nontrivial3} and  by Lemma  \ref{injectiveaction},  if $m\in \D\otimes_K  L^*$ and $\beta \cdot m=0$, then $m=0$. Also, since the characteristic of $K$ is coprime with $l$,  if $  (u^{l-1}+\ldots+u+1)  \cdot m=0$, we have that
$$m=\frac 1l \sum_{i=0}^{l-1}(1-u^i)m \in (u-1)\left (\D\otimes_K L^* \right ).$$ Thus,  $\ker \widetilde \gamma =\{0\}$. On the other hand by Claim \ref{flatness3}, $$\Tor_1^{K[G]}\left(  \left(  I_{K[A]}^G/\left ( I_{K[H]}^G\right ) \right )\oplus K[G]\ , \D\otimes_K  L^*\right ) $$ is trivial. Thus, $$\ \Tor_1^{K[G]}\left (   I_{K[G]}/\left (  I_{K[H]}^G \right )   ,\D\otimes_K  L^*\right )  =\{0\},$$ and so  $\D_{K[\overline G]}\otimes \left (I_{K[G]}/\left ( {}^G I_{K[H]}\right )\right )$ is flat.
\end{proof}

\subsection{Flat modules for free group algebras}
Observe that if a ring $T$ is a quotient of a ring $S$ and $M$ is a left flat $S$-module, then $T\otimes_S M$ is a left flat $T$-module. In the following theorem we show that in the case where the presentation complex of $G=\langle X|R\rangle$ is reducible without proper powers, the left flat $K[G]$-module constructed in Theorem \ref{key} can be lifted to a left flat $K[F]$-module.

\begin{teo}\label{key2}
 Let $K$ be a field, $F$  a free group, $W$ a strictly reducible subgroup of $F$ and put $G=F/\normal{W}$. Assume that $\D_{K[G]}$ exists. Then  the left $K [F]$-module $$\D_{K [G] }\otimes_K  \left  (I_{K [F]}/\left  ({}^FI_{K[W]}\right  ) \right  ) $$ is flat.
 \end{teo}

\begin{rem*} Notice that $$K[G]\otimes_{K[F]} \left  (\D_{K [G] }\otimes_K  \left (I_{K [F]}/\left ({}^FI_{K[W]}\right ) \right ) \right  ) \cong \D_{K [G] }\otimes_K  I_{K [G]}$$ as left $K[G]$-modules.
\end{rem*}

\begin{proof}
We prove the theorem by induction on the rank of $W$. If $W=\{1\}$, then the module $I_{K [F]}/I_{K[W]}^F\cong I_{K [F]}$ is free, and so, $\D_{K[F]}\otimes_K I_{K [F]}$ is free as well.
 
Now suppose that $F$ is generated by $X= X_1\sqcup \{x\}$, $W$ is generated by $R=R_1  \sqcup \{w\}$, the complex associated with the presentation $G_1=\langle X_1 |R_1 \rangle$ is   reducible without proper powers and  $w$ is either 1 or the image $\overline w$ of  $w$ within $G_1*\langle x \rangle $  is   not conjugate in to $G_1$ or $\langle x\rangle $  and
it  is not equal to a proper power in $G_1*\langle x \rangle$.
  We leave the case $w=1$ to the reader and  assume that $w\ne 1$.

We put $F_1=\langle X_1\rangle$, $F_2=\langle x\rangle $ and $W_1=\langle R_1\rangle$.
 We have that by the inductive hypothesis $\D_{K [G_1] }\otimes_K  \left (I_{K [F_1]}/\left ({}^{F_1}I_{K[W_1]}\right ) \right ) $ is flat as a left $K[F_1]$-module. 
\begin{claim} \label{flatness2}
The left $K[F]$-module
$\D_{K [G] }\otimes_K \left  (\left ({}^FI_{K [F_1]}\right ) /\left ({}^FI_{K[W_1]}\right ) \right  ) $ is flat.
\end{claim}
\begin{proof}
It is proved in the same way as Claim \ref{flatness}.
\end{proof}
Since the group $F$ is a free product of $F_1$ and $F_2$, we can write
$$w-1=\alpha_1+\ \beta\left (x-1\right )  \textrm {\ with\ }\alpha_1\in {}^FI_{F_1} \textrm {\ and\ }\beta\in K[F].$$
 
\begin{claim}   \label{nontrivial2}
If we write $\beta  =\displaystyle \sum_{i=1}^n c_i\cdot f_i\in K [F]\ \left (0\ne c_i\in K ,\ f_i\in F\right )$. Then
  all $g_i=f_i\normal{W}\in G$ are different.
\end{claim}
\begin{proof}
It is proved in the same way as Claim \ref{nontrivial}.
\end{proof} 

We will write $\D$ instead of $\D_{K[G]}$.
Let $L$ be a right $K [F]$-module. Then by Lemma \ref{dual},
$$\Tor_1^{K [F]}\left (L,\D\otimes_K  \left (I_{K[F]}/ \left ({}^FI_{K[W]}\right ) \right )\right ) \cong \Tor_1^{K [F]}\left ( I_{K[F]}/I_{K[W]}^F, \D\otimes_K  L^*\right ) .$$
 We have the following exact sequence of right $K[F]$-modules:
 $$0\to K[F]\xrightarrow{\gamma} \left (I_{K[F_1]}^F\right ) /\left (I_{K[W_1]}^F\right ) \oplus K[F]\to  I_{K[F]}/\left (I_{K[W]} ^F\right ) \to 0,$$
 where $\gamma (a) = \left (\alpha_1 a+I_{K[W_1]}^F, \beta a\right)$. Thus, we obtain  the exact sequence
\begin{multline*}
\Tor_1^{K[F]}\left (\left (I_{K[F_1]}^F\right ) /\left (I_{K[W_1]}^F\right ) \oplus  K[F], \D\otimes_K  L^*\right ) 
\to\\  \Tor_1^{K[F]}\left (I_{K[F]}/\left (I_{K[W]} ^F\right )  ,\D\otimes_K  L^*\right ) 
\to  \D\otimes_K  L^*\xrightarrow{\widetilde \gamma} \\ \left ( \left  (\left (I_{K[F_1]}^F\right ) /\left (I_{K[W_1]}^F\right ) \right  ) \otimes _{K[F]}\left  (  \D\otimes_K  L^*\right  ) \right  )
\oplus \left  ( \D\otimes_K  L^*\right  ) ,  \end{multline*}
where $\widetilde \gamma =\left (\widetilde \gamma_1, \widetilde \gamma_2\right ) $ and $\widetilde \gamma_2\left (m\right ) =\beta m$.
 
 By Claim \ref{nontrivial2}, $\beta$ satisfies the condition of  Lemma  \ref{injectiveaction}. Therefore, $\ker \widetilde \gamma =\{0\}$. On the other hand, $ \Tor_1^{K[F]}\left (\left (I_{K[F_1]}^F\right ) /\left (I_{K[W_1]}^F\right ) \oplus  K[F], \D\otimes_K  L^*\right ) $ is trivial by Claim \ref{flatness2}. Thus, the left $K [F]$-module $\D_{K [G] }\otimes_K  \left (I_{K [F]}/\left ({}^FI_{K[W]}\right ) \right ) $ is flat.\end{proof}

\section{Applications}
\label{sect:appl}

In this section we give some applications of our results.

\subsection{Mapping tori of free groups}

The first application we mention is a new proof of the coherence of an ascending HNN-extension of a free group.

\begin{cor}[\cite{FH99}]\label{FHth}
Let $G$ be  an ascending  HNN-extension of a free group. Then $G$ is coherent.
\end{cor}
\begin{proof}
By Theorem  \ref{mappingtorus}, $\Q[G]$ is coherent. Hence all finitely generated subgroups of $G$ are of type $\fp_2(\Q)$. Now the result follows from Theorem \ref{equivalence}.
\end{proof}

\subsection{Right angled Artin groups and Coxeter groups}\label{sec:raags}

Let $\Gamma$ be a simplicial graph, then the {\bf right angled Artin group} (RAAG) $A(\Gamma)$ is the group with presentation
\[
A(\Gamma) = \langle V(\Gamma) \mid [v, w] = 1, \text{ if $(v, w)\in E(\Gamma)$}\rangle.
\]
The classification of coherent right angled Artin groups was carried out by Droms \cite{droms_87}. Despite the simple description of the presentation of a right angled Artin group, their subgroup structure is extremely rich. Indeed, the first examples of groups of type $\fp(\Z)$ that are not finitely presented, due to Bestvina--Brady, are subgroups of RAAGs \cite{bestvina_97}. Although the Bestvina--Brady groups are homologically finitely presented, they must be homologically incoherent.

\begin{pro}
\label{raags}
Let $k$ be a ring and let $G$ be a subgroup of a right angled Artin group. If every finitely generated subgroup of $G$ is of type $\fp_2(k)$, then $G$ is coherent. 
\end{pro}

\begin{proof}
It suffices to prove the result in the case that $G$ is finitely generated. Hence we may assume that $G\le  A(\Gamma)$ where $\Gamma$ is a finite simplicial graph, not necessarily connected. If $v\in V(\Gamma)$ is a vertex, denote by $\Lambda_v\subset V(\Gamma)$ the subgraph on the set of vertices adjacent to $v$ and by $\Gamma_v\subset \Gamma$ the subgraph on the vertices $V(\Gamma) - \{v\}$. It is clear that $A(\Gamma) \isom A(\Gamma_v)*_{\psi}$ where $\psi$ is the identity isomorphism on $A(\Lambda_v)$. By induction, we see that there is a sequence of subgraphs $\Gamma_0\subset \ldots\subset \Gamma_n = \Gamma$ such that $A(\Gamma_0)\isom \Z$ and such that $A(\Gamma_i)$ splits as a HNN-extension over $A(\Gamma_{i-1})$ for all $i\geqslant 1$. Now the result follows from \cref{hierarchy}.
\end{proof}

If $\Gamma$ is a simplicial graph and $m\colon E(\Gamma)\to \N_{\geqslant 2}$ is a map, the {\bf Coxeter group} $C(\Gamma)$ is the group given by the presentation
\[
C(\Gamma) = \langle V(\Gamma) \mid v_i^2, (v_iv_j)^{m(e)}, e = \{v_i, v_j\}\in E(\Gamma)\rangle.
\]
A Coxeter subgroup of $C(\Gamma)$ is a subgroup generated by a subset of the vertex generators of $C(\Gamma)$. This subgroup will be isomorphic to the Coxeter group on the full subgraph containing these vertices. If $G = C(\Gamma)$, following \cite[Definition 11.25]{Wi20}, define:  
\[
{\chi}(G) = 1 - \frac{|V(\Gamma)|}2 + \sum_{e\in E(\Gamma)}\frac{1}{2m(e)}.
\]
Many Coxeter groups were shown to be incoherent in \cite{jankiewicz_16}. We may apply our techniques to establish coherence of a large subclass of Coxeter groups. This is one direction of \cite[Conjecture 4.7]{jankiewicz_16}  (see also \cite[Conjectures 9.29 \& 11.29]{Wi20}) and is \cref{coxeter} from the introduction.

\begin{teo}
Let $G$ be a Coxeter group and suppose that ${\chi}(H)\leqslant 0$ for each Coxeter subgroup $H\le  G$ generated by at least three elements. Then $G$ is coherent.
\end{teo}

\begin{proof}
By  \cite[Theorem 4.5]{jankiewicz_16} and \cite[Theorems 11.12]{Wi20}, there is a two-complex $X$ with non-positive immersions such that $\pi_1(X)$ is a finite index subgroup of $G$. Hence, $G$ is homologically coherent by \cref{npi}. By work of Haglund--Wise \cite[Corollary 1.3]{haglund_10}, $G$ has a finite index subgroup that is a subgroup of a right angled Artin group. Applying \cref{raags} yields the result.
\end{proof}

\subsection{Groups with staggered presentations}

The aim of this subsection is to prove \cref{staggered_coherent}. 
Since the presentation complex of a staggered presentation is a bireducible complex, \cref{staggered_coherent} is a corollary of the following more general theorem.

\begin{teo}
\label{bired_coherent}
If $X$ is a finite bireducible complex, then $\pi_1(X)$ is coherent.
\end{teo}

A finitely generated subgroup $H\le  G$ has the {\bf finitely generated intersection property (f.g.i.p.)} if $H\cap K$ is finitely generated for all finitely generated subgroups $K\le  G$. A key ingredient in the proof of \cref{bired_coherent} is the following theorem, due to Karrass--Solitar \cite{KS70,KS71}. See also \cite[Theorem 5.4]{Wi20}.

\begin{teo}
\label{fgip_coherent}
If $G$ splits as a graph of groups whose vertex groups are coherent and whose edge groups have f.g.i.p., then $G$ is coherent.
\end{teo}

In order to apply this theorem to bireducible two-complexes, we need to show that Magnus subgroups have f.g.i.p.

\begin{teo}\label{Magnus}
Magnus subgroups of fundamental groups of finite bireducible complexes have f.g.i.p.
\end{teo}

We will give two proofs of this theorem. The first proof uses Theorem \ref{key3} and the second uses the fact that bireducible complexes have NTPI.

\begin{proof}[The first proof of Theorem \ref{Magnus}]
Let $H$ be a Magnus subgroup of the fundamental group of  a bireducible two-complex $X$. Let $U$ be a finitely generated subgroup of $G=\pi_1(X)$. For an arbitrary subgroup $K$ of $G$ let $$
\overline d(K)=\left \{ \begin{array}{ cc}0 & \textrm{if\ }  K=\{1\} \\ d(K)-1	& \textrm{if\ }  K\not =\{1\} \end{array}\right ..$$
We aim to show that
\begin{equation}\label{stronginert}
\sum_{g \in U \backslash G / H} \overline{d}(H \cap U^g) \leqslant \overline{d}(U).
\end{equation}
In particular, this implies that $d(H\cap U) \leqslant  d(U)$. 

As in Subsection \ref{sub:Magnus} we put $\overline G=\pi_1(\overline X)$ and $\D=\D_{\Q[\overline G]}$. Denote by $N$ the kernel of the canonical map $G\to \overline G$. 
 Observe that  $N\cap H=\{1\}$. Thus, the inequality (\ref{stronginert}) holds if $U\le N$. From now on, we will assume that $U\not \le N$.

For any subgroup $K$ of $G$ 
let $$\delta_K=\left \{ \begin{array}{ cc} 1 & \textrm{if\ }  K\le N\\0 & \textrm{if\ }  K\not \le N\end{array}\right . .$$
Taking into account the exact sequence
$$0\to  \Tor_1^{\Q[G]}(\D, \Q[G/ K ])\to \D\otimes_{\Q[K]} I_{\Q[K]}\to \D\otimes_{\Q[G]} {\Q[G]}\to  \D\otimes_{\Q[G]}  \Q[G/ K ]\to 0,$$
we obtain the following bound on the minimal number of generators of $K$:
\begin{multline*}
\overline d(K) \geqslant  \dim _{\D} \left ( \D\otimes_{\Q[K]} I_{\Q[K]} \right)  -1=\\  \dim _{\D}  \Tor_1^{\Q[G]}(\D, \Q[G/ K ])-\dim _{\D} \left( \D\otimes_{\Q[G]}  \Q[G/ K ]\right )  =\\ \dim _{\D}  \Tor_1^{\Q[G]}(\D, \Q[G/ K ])-\delta_K,
\end{multline*}
and we have the equality if $K$ is free and non-trivial.
Since $H$ is free and $N\cap H=\{1\}$, we have
$$\overline d(H\cap U^g) = \dim _{\D} \Tor_1^{\Q[G]}(\D, \Q[G/\left (H\cap U^g\right) ]) \textrm{\ for every\ }g\in G.$$
  Observe that we have the following isomorphism of $\Q[G]$-modules:
 $$\Q[G/H]\otimes_\Q \Q[G/U]\cong \oplus_{g\in U\backslash G/H}\Q[G/(H\cap U^g)].$$
  Therefore,  we obtain  that 
$$\sum_{g\in U\backslash G/H}\dim _{\D} \Tor_1^{\Q[G]}(\D, \Q[G/\left (H\cap U^g\right) ])= \dim_{\D} \Tor_1^{\Q[G]} (\D, \Q[G/H]\otimes_\Q \Q[G/U]).$$
On the other hand, we have the following  exact sequence
\begin{multline*}\Tor_1 ^{\Q[G]}(\D,  (I_{\Q[G]}/{}^GI_{\Q[H]})\otimes_{\Q}\Q[G/U])\to \\ \Tor_1^{\Q[G]} (\D, \Q[G/H]\otimes_\Q \Q[G/U])\to \Tor_1 ^{\Q[G]}(\D,  \Q[G/U]).\end{multline*}
By Lemma \ref{dual} and Theorem \ref{key3}, $\Tor_1 ^{\Q[G]}(\D,  (I_{\Q[G]}/{}^GI_{\Q[H]})\otimes_{\Q}\Q[G/U])=0$. Thus,
$$\sum_{g\in U\backslash G/H} \overline d(H\cap U^g)  \leqslant     \dim_{\D} \Tor_1 ^{\Q[G]}(\D,  \Q[G/U])  \leqslant   \overline d(U)+\delta_U=\overline d(U). $$
This proves (\ref{stronginert}).
\end{proof}

\begin{rem*}
If $G$ is a group and $H\le  G$ is a subgroup, say $H$ is {\bf inert} ({\bf strongly inert}) in $G$ if $d(K\cap H)\leqslant d(K)$  ($\sum_{g\in K\backslash G/H} \overline d(H\cap K^g)  \leqslant   \overline d(K)$  ) for every subgroup $K\le  G$. Inert subgroups were first defined by Dicks--Ventura in the context of free groups \cite{DV96}. The above proof of \cref{Magnus} yields a stronger result: Magnus subgroups of fundamental groups of finite bireducible two-complexes are inert and strongly inert. In particular, Magnus subgroups of one-relator groups are inert and strongly inert.
\end{rem*}

\begin{proof}[The second  proof of Theorem \ref{Magnus}]  
Let $M=\pi_1(\Lambda)$ be a Magnus subgroup of the fundamental group of a bireducible two-complex $X$, where $\Lambda$ is connected and small.
We claim that
\[
Y = X\cup_{\Lambda}{X^\prime}
\]
is bireducible, where ${X^\prime}$ is an isomorphic copy of $X$. Indeed let $U\subset Y$ be a subcomplex which we may assume contains at least one two-cell in $X$ and at least one two-cell in ${X^\prime}$. By assumption, there exists a reduction of $(U\cap X)\cup \Lambda$ containing $\Lambda$  and so there is a reducing edge in $(U\cap X) - \Lambda$. Similarly for $U\cap {X^\prime}$. Hence $Y$ is bireducible. 

If $H\le \pi_1(X)$ is a finitely generated subgroup, then we have
\[
G = H*_{H\cap M}{H^\prime}\le \pi_1(Y)
\]
where ${H^\prime}$ is an isomorphic copy of $H$. Since $G$ is finitely generated, we see that $b_1(G)<\infty$. Applying the Mayer--Vietoris sequence due to Swan \cite[Theorem 2.3]{swan_69} to the amalgamated free product decomposition of $G$, we see that $b_1(H\cap U)<\infty$ by \cref{2betti}. As $M$ is free, $H\cap M$ is free also and so must be finitely generated. We have proved that $M$ has the finitely generated intersection property in $\pi_1(X)$.  
\end{proof}

\begin{proof}[Proof of Theorem  \ref{bired_coherent}]The proof is by induction on the number of two-cells in $X$. The base case is \cref{main}(1). Now suppose that $X$ contains at least two two-cells and assume the inductive hypothesis.

Let $Z\subset X$ be a reduction and let $\alpha$ be the two-cell in $X-Z$. Denote by $X_{\alpha}\subset X$ the smallest subcomplex containing $\alpha$. We have
\[
X = Z\cup_{\Lambda}X_{\alpha}
\]
where $\Lambda = Z\cap X_{\alpha}$. Since $X$ is bireducible, $\Lambda$ is small. We may add edges to $X$ and $X_{\alpha}$ to ensure that $Z$ and $\Lambda$ are connected as the resulting two-complex will have coherent fundamental group if and only if the original one does. \cref{injective_sub} tells us that $\pi_1(\Lambda)$, $\pi_1(Z)$ and $\pi_1(X_{\alpha})$ are subgroups of $\pi_1(X)$. In particular, we have
\[
\pi_1(X) \isom \pi_1(Z)*_{\pi_1(\Lambda)}\pi_1(X_{\alpha}).
\]
In Theorem \ref{Magnus}, we have proved that $\pi_1(\Lambda)$ has the finitely generated intersection property in $\pi_1(X)$. Since $\pi_1(Z)$ and $\pi_1(X_{\alpha})$ are coherent by induction, we apply Theorem \ref{fgip_coherent} to obtain the result.
\end{proof}

\begin{proof}[Proof of \cref{staggered_coherent}]
Let $G$ be a group with a staggered presentation. If $G$ is finitely presented, the result follows from \cref{bired_coherent}. Now suppose that $G$ has an infinite staggered presentation 
\[
\langle S, \ldots, x_{-1}, x_0, x_1, \ldots \mid \ldots, r_{-1}, r_0, r_1, \ldots\rangle
\]
where the ordering of the generators is given by their indexing and $S$ is the set of unordered generators. Denote by $m_i$ and $M_i$ the smallest and largest integers such that $r_i$ mentions $x_{m_i}$ and $x_{M_i}$ respectively. Let $H_i = \langle S, x_{m_i}, \ldots, x_{M_i} \mid r_i\rangle$ and $A_i = F(S, x_{m_{i+1}}, \ldots, x_{M_i})$. By the Freiheitssatz \cite{magnus_30} we see that
\[
G \isom \ldots \underset{A_{i-1}}{*}H_i\underset{A_i}{*}H_{i+1}\underset{A_{i+1}}{*}\ldots
\]
If $H\le  G$ is a finitely generated subgroup, there exist integers $i\le  j$ such that
\[
H\le  H_i\underset{A_i}{*}\ldots\underset{A_{j-1}}{*}H_j\le  G
\]
In particular, $H$ is a finitely generated subgroup of the fundamental group of a bireducible complex with finitely many two-cells. Now \cref{bired_coherent} finishes the proof.
\end{proof}

We close this section with an example of a finitely generated group with a staggered presentation which is not virtually torsion-free, demonstrating that we could not simply appeal to \cref{npi} in the proof of \cref{staggered_coherent}.

\begin{exa}\label{staggered_example}
Consider the following group known as the Baumslag--Gersten group:
\[
G = \langle a, t \mid [a^t,a^{-1}] = a\rangle.
\]
Baumslag proved that every finite quotient of $G$ is finite cyclic \cite{baumslag_69}. Hence, if $G\to H$ is a homomorphism to a finite group, $a$ is in the kernel as $a\in [G, G]$. So now consider the following example of a group with a staggered presentation:
\[
K = \langle a, t \mid [a^t,a^{-1}] = a\rangle *_{\langle a\rangle = \langle b\rangle}\langle b, c \mid [b, c]^2\rangle
\]
Let $\phi\colon K\to H$ be a homomorphism with $H$ finite. By the above, we see that $a\in \ker(\phi)$ and thus $b\in \ker(\phi)$. But this implies that $\phi([b, c]) = 1$ and so $\ker(\phi)$ has elements of finite order. Since $\phi$ was arbitrary, this shows that $K$ is not virtually torsion-free.
\end{exa}

\subsection{An alternative proof of the rank one Hanna Neumann conjecture}

Given two finitely generated subgroups $U$ and $W$ of a free group $F$  the Friedman--Mineyev theorem  \cite{Fr14, Mi12} (previously known as the Strengthened  Hanna Neumann conjecture)  states  that
$$\sum_{x\in W \backslash  F/U} \overline d \left (xUx^{-1} \cap W\right ) \leqslant \overline d\left (U\right )  \overline d\left (W\right ) .$$ 
This result   says nothing about the cyclic intersections  $xUx^{-1} \cap W$.  Wise proposed a rank-1 version of this conjecture, which was proved independently by Helfer and Wise \cite{HW16} and Louder and Wilton \cite{LW17}.

\begin{teo}[Helfer-Wise, Louder-Wilton]  \label{1rank} Let $U$ be a subgroup of a free group $F$,  $w\in F$ an element that is not a proper power and $W=\langle w\rangle$. Then
$$\sum_{x\in W \backslash  F/U}  d \left (xUx^{-1} \cap W\right ) \leqslant 
  \left \{
   \begin{array}{cc} 
   d\left (U\right )  & \textrm{if\ }  U\le \normal{W}\\
 d\left (U\right ) -1 & \textrm{if\ } U\not \le \normal{W} \end{array}\right . .
$$ 
\end{teo}
In this section we will prove a generalisation of this result.  In order to formulate it, we need another interpretation for the sum $\sum_{x\in W \backslash  F/U}  d (xUx^{-1} \cap W)$. We follow the approach developed in \cite{Ja17, AJ22}. Let $W$ be an arbitrary subgroup of $F$.
Consider $\Q[F/U]$ as a left $\Q[W]$-module. Then
$$\Q[F/U]\cong \bigoplus_{x\in W \backslash  F/U}  \Q[W/(xUx^{-1}\cap W)].$$
The exact sequence
\[
0 \to {}^WI_{\Q[xUx^{-1}\cap W]}\to \Q[W] \to \Q[W/(xUx^{-1}\cap W)]\to 0
\]
yields a long exact sequence with extract
\[
\Tor_1^{\Q[W]}(\Q, \Q[W]) \to \Tor_1^{\Q[W]}(\Q, \Q[W/(xUx^{-1}\cap W)]) \to \Q\otimes_{\Q}{}^WI_{\Q[xUx^{-1}\cap W]}\to \Q\otimes_{\Q}\Q[W].
\]
Since $\Tor_1^{\Q[W]}(\Q, \Q[W]) = 0$ and the map on the right is the zero map, this implies that
\[
d(xUx^{-1}\cap W)= \dim_{\Q}(\Q\otimes_{\Q}{}^WI_{\Q[xUx^{-1}\cap W]}) =\dim_{\Q}  \Tor_1^{\Q[W]}(\Q,\Q[W/(xUx^{-1}\cap W)]).
\]
Therefore, the sum   that appears in Theorem \ref{1rank} has the following interpretation
$$ \sum_{x\in W \backslash  F/U}  d (xUx^{-1} \cap W)=\dim_{\Q}  \Tor_1^{\Q[W]}(\Q, \Q[F/U]).$$

If $F$ is a free group, all left ideals of ${\Q} [F]$ are left free ${\Q} [F]$-modules of a unique rank. We denote by $\rk  (L) $ the rank of a free ${\Q} [F]$-module $L$.
Notice that $$\Q[F/U]\cong \Q[F]/\left ({}^FI_{\Q[U]}\right )\textrm{\ and\ } \rk \left ({}^FI_{ \Q[U]}\right ) =d\left (U\right ) .$$
Thus, our next  result can be viewed as a generalization of Theorem \ref{1rank} and Theorem \ref{HN}.  

\begin{teo}\label{1rank_2}
Let $F$ be a free group, $W$ a strictly reducible  subgroup of $F$, $I$ the ideal of ${\Q} [F]$ generated by $\{w-1\colon w\in W\}$ and $L$ a left ideal of ${\Q} [F]$. Then  we have the following inequality
$$\dim_{{\Q} } \Tor_1^{\Q[W]}(\Q,\Q[F]/L)  \leqslant    \left \{
   \begin{array}{cc} 
 \rk (L)  & \textrm{if\ } L\le I\\
  \rk (L )-1  & \textrm{if\ } L\not \le I  \end{array}\right . .
$$ 
\end{teo}

\begin{proof}
Let us put $G=F/\normal{W}$, $\D=\D_{{\Q} [G]}$ and $M={\Q} [F]/L$.
\begin{claim}\label{claim:iso}
The induced  left $\Q[F]$-module $\Q[F]\otimes_{\Q[W]} M$ is isomorphic to the left module $\Q[F/W]\otimes_{\Q} M$.
\end{claim}
\begin{proof} Let $T$ be a left transversal of $F$ with respect to $W$.
Define  a $\Q$-linear map $\tau: \Q[F]\otimes_{\Q[W]} M\to \Q[F/W]\otimes_{\Q}M$ by $$ \ \tau: t\otimes  m\mapsto tW\otimes  tm\ (t\in T,  m\in M).$$  It is bijective because we can define the inverse map by $\tau^{-1}(tW\otimes m)=t\otimes t^{-1}m$.

If $f\in F$ and $t\in T$, then there are $t^\prime \in T$ and $w\in W$ such that $ft=t^\prime w$. Thus, we obtain
\begin{align*}\tau(f(t\otimes m))=\tau(ft\otimes m)=\tau(t^\prime w\otimes m)=\tau (t^\prime \otimes wm)=\\ t^\prime W \otimes t^\prime w m=ftW\otimes ft m=f\left ( tW\otimes tm\right )=f\tau(t\otimes m).\end{align*}
Therefore, $\tau$ is also a $\Q[F]$-homomorphism.
\end{proof}

Note that 
\begin{multline*}
 \dim_{{\Q} } \Tor^{{\Q} [W]}_1({\Q} ,M)= \dim_{\D}\Tor^{{\Q} [W]}_1(\D,M)\myeq{Shapiro's lemma}\\
    \dim_{\D} \Tor_1^{{\Q} [F]}(\D,{\Q} [F]\otimes_{\Q[W]}  M)\myeq{Claim \ref{claim:iso}}\dim_{\D} \Tor_1^{{\Q} [F]}(\D,{\Q} [F/W]\otimes_{\Q}  M).\end{multline*}
Taking into account the exact sequence of left ${\Q} [F]$-modules
$$0\to \left (I_{{\Q} [F]}/\left ({}^FI_{\Q[W]}\right )\right )\otimes_{\Q}  M \to {\Q} [F/W]\otimes_{\Q}  M\to M \to 0,$$
we obtain that
\begin{multline}\label{sum}
 \dim_{\D} \Tor_1^{{\Q} [F]}\left (\D,{\Q} [F/W]\otimes_{\Q}  M \right ) \leqslant \\
 \dim_{\D} \Tor_1^{{\Q} [F]} \left (\D,  \left (I_{{\Q} [F]}/\left ({}^FI_{\Q[W]}\right )\right )\otimes_{\Q}  M\right )+ \dim_{\D} \Tor_1^{{\Q} [F]}(\D, M).\end{multline}
Observe that      if we see $\D$ as a right ${\Q} [F]$-module, $\D^*\cong \D$. Thus, from Lemma \ref{dual} and Theorem \ref{key2}, we obtain that 
\begin{multline}\label{vanish}
 \Tor_1^{{\Q} [F]}\left (\D,  \left (I_{{\Q} [F]}/\left ({}^FI_{\Q[W]}\right )\right )\otimes_{\Q}  M\right ) = \\
 \Tor_1^{{\Q} [F]} \left ( M^*, \D\otimes_{\Q}   \left (I_{{\Q} [F]}/\left ({}^FI_{\Q[W]}\right )\right)\right )=\{0\}.\end{multline}
Since $\D\otimes_{\Q}M = 0$ if and only if $L\not\le I$, the second summand in (\ref{sum}) we have
that
\begin{multline*} \dim_{\D} \Tor_1^{{\Q} [F]}(\D, M)= \rk(L)-1+\dim_{\D}  \Tor_0^{{\Q} [F]}(\D, M) =\\
 \rk(L)-1+\left \{
   \begin{array}{cc} 
1& \textrm{if\ } L\le I\\
0 & \textrm{if\ } L\not \le I  \end{array}\right .=
\left \{
   \begin{array}{cc} 
 \rk (L) & \textrm{if\ } L\le I\\
  \rk(L)-1 & \textrm{if\ } L\not \le I  \end{array}\right . .
\end{multline*}
This finishes the proof of  the theorem.
\end{proof}

\section{Further comments and open questions} \label{sect:final}

In \cite[Sections 16 \& 17]{Wi20} Wise studied the connection between the non-positive immersions property and the vanishing of the second $L^2$-Betti number of a two-complex (as Wise indicated it was a suggestion of Gromov). Theorem \ref{npi} points again in the same direction. We want to go further and ask the following question.

\begin{Question} Is it true that a compact aspherical two-complex has  non-positive immersions if and only if it has trivial second $L^2$-Betti number? \end{Question}

In the introduction we have mentioned the conjecture of Wise \cite[Conjecture 12.11]{Wi20} that predicts that fundamental groups of two-complexes with non-positive immersions are coherent. In the same spirit we propose the following conjecture.

\begin{Conj}
\label{betti_conjecture}
Let $G$ be the fundamental group of an aspherical 2-complex with trivial  second $L^2$-Betti number. Then $G$ is coherent.
\end{Conj}

Note that in the case that $G$ is known to satisfy the strong Atiyah conjecture, the homological version of \cref{betti_conjecture} holds by \cref{Atiyah_locally} and \cref{npi}. This also leads us to the following question.

\begin{Question}
\label{hom_coherence_question}
Let $G$ be a homologically coherent group. Is $G$ coherent? What if $G$ has cohomological dimension two?
\end{Question}

In the case that $G$ is hyperbolic and has cohomological dimension two, then the answer is yes by a theorem of Gersten \cite{Ge96}. If instead $G$ admits a hierarchy terminating at coherent groups, then the answer is also yes by \cref{hierarchy}.

The relation between the coherence and the vanishing of the  second $L^2$-Betti number suggests also the following question.
\begin{Question}
Let $G$ be the fundamental group of an aspherical 2-complex. Assume that  $G$ is coherent. Is it true that $G$ has trivial  second $L^2$-Betti number?
\end{Question}

For this question we are more inclined towards the answer being no. It would be interesting to calculate the $L^2$-Betti numbers of Coxeter groups whose Davis complex has dimension 2 and compare the results with  predictions of the conjecture of Jankievicz--Wise \cite[Conjecture 4.7]{jankiewicz_16}. Another possible source of examples is torsion-free subgroups of finite index of the groups $\SL_2(\Z[1/p])$. The question of their coherence was posed by Serre in \cite[page 734-735]{Ca73}. However, since these groups  are lattices in $\SL_2(\R)\times \SL_2(\Q_p)$ they are measurable equivalent to the product of two non-abelian free groups $F_2\times F_2$, and so, by a result of Gaboriau \cite[Theorem 6.3]{Ga02}, have non-trivial second $L^2$-Betti numbers. Higman's group is another potential candidate whose coherence or incoherence is not yet known (see \cite[Problem 20]{Wi20}). Its presentation complex is aspherical and has positive Euler characteristic and so has non-vanishing second $L^2$-Betti number.

Similar conjectures and questions can be formulated about coherence of group algebras. We only mention one.
In view of Theorem \ref{npi} and by analogy with Wise's conjecture, we propose the following conjecture.

\begin{Conj}\label{conj:groupalgebracoherent}
Let $K$ be a field and $G$ the fundamental group of   a two-complex with non-positive immersions. Then the group algebra $K[G]$ is coherent.
\end{Conj}

\bibliographystyle{amsalpha}
\bibliography{bibliography}

\providecommand{\bysame}{\leavevmode\hbox to3em{\hrulefill}\thinspace}
\providecommand{\MR}{\relax\ifhmode\unskip\space\fi MR }
\providecommand{\MRhref}[2]{%
  \href{http://www.ams.org/mathscinet-getitem?mr=#1}{#2}
}
\providecommand{\href}[2]{#2}
\begin{thebibliography}{HKKL24}

\bibitem[\.{A}82]{Ab82}
Hans \.{A}berg, \emph{Coherence of amalgamations}, J. Algebra \textbf{78}
  (1982), no.~2, 372--385. \MR{680365}

\bibitem[AJZ22]{AJ22}
Yago Antol\'{\i}n and Andrei Jaikin-Zapirain, \emph{The {H}anna {N}eumann
  conjecture for surface groups}, Compos. Math. \textbf{158} (2022), no.~9,
  1850--1877. \MR{4496360}

\bibitem[Bau69]{baumslag_69}
Gilbert Baumslag, \emph{A non-cyclic one-relator group all of whose finite
  quotients are cyclic}, J. Austral. Math. Soc. \textbf{10} (1969), 497--498.
  \MR{0254127}

\bibitem[Bau74]{Ba74}
\bysame, \emph{Some problems on one-relator groups}, Proceedings of the
  {S}econd {I}nternational {C}onference on the {T}heory of {G}roups
  ({A}ustralian {N}at. {U}niv., {C}anberra, 1973), 1974, pp.~75--81. Lecture
  Notes in Math., Vol. 372. \MR{0364463}

\bibitem[BB97]{bestvina_97}
Mladen Bestvina and Noel Brady, \emph{Morse theory and finiteness properties of
  groups}, Invent. Math. \textbf{129} (1997), no.~3, 445--470. \MR{1465330}

\bibitem[Ber19]{Be18}
George~M. Bergman, \emph{Some results relevant to embeddability of rings
  (especially group algebras) in division rings}, J. Algebra \textbf{535}
  (2019), 503--540. \MR{3982005}

\bibitem[Bro84]{Br84}
S.~D. Brodski\u{\i}, \emph{Equations over groups, and groups with one defining
  relation}, Sibirsk. Mat. Zh. \textbf{25} (1984), no.~2, 84--103. \MR{741011}

\bibitem[Bro94]{Br82}
Kenneth~S. Brown, \emph{Cohomology of groups}, Graduate Texts in Mathematics,
  vol.~87, Springer-Verlag, New York, 1994, Corrected reprint of the 1982
  original. \MR{1324339}

\bibitem[BS78]{bieri_78}
Robert Bieri and Ralph Strebel, \emph{Almost finitely presented soluble
  groups}, Comment. Math. Helv. \textbf{53} (1978), no.~2, 258--278.
  \MR{498863}

\bibitem[BS79]{BS79}
\bysame, \emph{Soluble groups with coherent group rings}, Homological group
  theory ({P}roc. {S}ympos., {D}urham, 1977), London Math. Soc. Lecture Note
  Ser., vol.~36, Cambridge Univ. Press, Cambridge-New York, 1979, pp.~235--240.
  \MR{564427}

\bibitem[BS90]{baumslag_90}
Gilbert Baumslag and Peter~B. Shalen, \emph{Amalgamated products and finitely
  presented groups}, Comment. Math. Helv. \textbf{65} (1990), no.~2, 243--254.
  \MR{1057242}

\bibitem[BZ15]{BZ15}
Alexey Bondal and Ilya Zhdanovskiy, \emph{Coherence of relatively quasi-free
  algebras}, Eur. J. Math. \textbf{1} (2015), no.~4, 695--703. \MR{3426174}

\bibitem[DD89]{dicks_89}
Warren Dicks and M.~J. Dunwoody, \emph{Groups acting on graphs}, Cambridge
  Studies in Advanced Mathematics, vol.~17, Cambridge University Press,
  Cambridge, 1989. \MR{1001965}

\bibitem[DHS04]{DHS04}
Warren Dicks, Dolors Herbera, and Javier S\'{a}nchez, \emph{On a theorem of
  {I}an {H}ughes about division rings of fractions}, Comm. Algebra \textbf{32}
  (2004), no.~3, 1127--1149. \MR{2063801}

\bibitem[DL07]{DL07}
Warren Dicks and Peter~A. Linnell, \emph{{$L^2$}-{B}etti numbers of one-relator
  groups}, Math. Ann. \textbf{337} (2007), no.~4, 855--874. \MR{2285740}

\bibitem[Dro87]{droms_87}
Carl Droms, \emph{Graph groups, coherence, and three-manifolds}, J. Algebra
  \textbf{106} (1987), no.~2, 484--489. \MR{880971}

\bibitem[Dub87]{Du87}
N.~I. Dubrovin, \emph{Invertibility of the group ring of a right-ordered group
  over a division ring}, Mat. Zametki \textbf{42} (1987), no.~4, 508--518, 622.
  \MR{917804}

\bibitem[Dun85]{dunwoody_85}
M.~J. Dunwoody, \emph{The accessibility of finitely presented groups}, Invent.
  Math. \textbf{81} (1985), no.~3, 449--457. \MR{807066}

\bibitem[DV96]{DV96}
Warren Dicks and Enric Ventura, \emph{The group fixed by a family of injective
  endomorphisms of a free group}, Contemporary Mathematics, vol. 195, American
  Mathematical Society, Providence, RI, 1996. \MR{1385923}

\bibitem[Eck01]{Ec01}
Beno Eckmann, \emph{Idempotents in a complex group algebra, projective modules,
  and the von {N}eumann algebra}, Arch. Math. (Basel) \textbf{76} (2001),
  no.~4, 241--249. \MR{1825003}

\bibitem[FH99]{FH99}
Mark Feighn and Michael Handel, \emph{Mapping tori of free group automorphisms
  are coherent}, Ann. of Math. (2) \textbf{149} (1999), no.~3, 1061--1077.
  \MR{1709311}

\bibitem[FKS72]{FKS72}
J.~Fischer, A.~Karrass, and D.~Solitar, \emph{On one-relator groups having
  elements of finite order}, Proc. Amer. Math. Soc. \textbf{33} (1972),
  297--301. \MR{311780}

\bibitem[FP06]{fujiwara_06}
K.~Fujiwara and P.~Papasoglu, \emph{J{SJ}-decompositions of finitely presented
  groups and complexes of groups}, Geom. Funct. Anal. \textbf{16} (2006),
  no.~1, 70--125. \MR{2221253}

\bibitem[Fri15]{Fr14}
Joel Friedman, \emph{Sheaves on graphs, their homological invariants, and a
  proof of the {H}anna {N}eumann conjecture: with an appendix by {W}arren
  {D}icks}, Mem. Amer. Math. Soc. \textbf{233} (2015), no.~1100, xii+106, With
  an appendix by Warren Dicks. \MR{3289057}

\bibitem[Gab02]{Ga02}
Damien Gaboriau, \emph{Invariants {$l^2$} de relations d'\'{e}quivalence et de
  groupes}, Publ. Math. Inst. Hautes \'{E}tudes Sci. (2002), no.~95, 93--150.
  \MR{1953191}

\bibitem[Ger96]{Ge96}
S.~M. Gersten, \emph{Subgroups of word hyperbolic groups in dimension {$2$}},
  J. London Math. Soc. (2) \textbf{54} (1996), no.~2, 261--283. \MR{1405055}

\bibitem[GL17]{guirardel_17}
Vincent Guirardel and Gilbert Levitt, \emph{J{SJ} decompositions of groups},
  Ast\'{e}risque (2017), no.~395, vii+165. \MR{3758992}

\bibitem[GN21]{GN21}
Damien Gaboriau and Camille No\^{u}s, \emph{On the top-dimensional
  {$\ell^2$}-{B}etti numbers}, Ann. Fac. Sci. Toulouse Math. (6) \textbf{30}
  (2021), no.~5, 1121--1137. \MR{4401387}

\bibitem[Gra20]{Gr20}
Joachim Grater, \emph{Free division rings of fractions of crossed products of
  groups with {C}onradian left-orders}, Forum Math. \textbf{32} (2020), no.~3,
  739--772. \MR{4095506}

\bibitem[GS15]{GS15}
Joachim Grater and Robert~P. Sperner, \emph{On embedding left-ordered groups
  into division rings}, Forum Math. \textbf{27} (2015), no.~1, 485--518.
  \MR{3334070}

\bibitem[HKKL24]{HKKL23}
Sam Hughes, Dawid Kielak, Peter~H. Kropholler, and Ian~J. Leary,
  \emph{Coherence for elementary amenable groups}, Proc. Amer. Math. Soc.
  \textbf{152} (2024), no.~3, 977--986. \MR{4693660}

\bibitem[HLA22]{HL22}
Fabian Henneke and Diego L\'{o}pez-\'{A}lvarez, \emph{Pseudo-{S}ylvester
  domains and skew {L}aurent polynomials over firs}, J. Algebra Appl.
  \textbf{21} (2022), no.~8, Paper No. 2250168, 28. \MR{4469346}

\bibitem[How81]{howie_81}
James Howie, \emph{On pairs of {$2$}-complexes and systems of equations over
  groups}, J. Reine Angew. Math. \textbf{324} (1981), 165--174. \MR{614523}

\bibitem[How82]{Ho82}
\bysame, \emph{On locally indicable groups}, Math. Z. \textbf{180} (1982),
  no.~4, 445--461. \MR{667000}

\bibitem[How84]{Ho84}
\bysame, \emph{Cohomology of one-relator products of locally indicable groups},
  J. London Math. Soc. (2) \textbf{30} (1984), no.~3, 419--430. \MR{810951}

\bibitem[How87]{howie_84}
\bysame, \emph{How to generalize one-relator group theory}, Combinatorial group
  theory and topology ({A}lta, {U}tah, 1984), Ann. of Math. Stud., vol. 111,
  Princeton Univ. Press, Princeton, NJ, 1987, pp.~53--78. \MR{895609}

\bibitem[How00]{Ho00}
\bysame, \emph{A short proof of a theorem of {B}rodski\u{\i}}, Publ. Mat.
  \textbf{44} (2000), no.~2, 641--647. \MR{1800825}

\bibitem[HS20]{howie_20}
James Howie and Hamish Short, \emph{Coherence and one-relator products of
  locally indicable groups}, 2020.

\bibitem[Hug70]{Hu70}
Ian Hughes, \emph{Division rings of fractions for group rings}, Comm. Pure
  Appl. Math. \textbf{23} (1970), 181--188. \MR{263934}

\bibitem[HW10]{haglund_10}
Fr\'{e}d\'{e}ric Haglund and Daniel~T. Wise, \emph{Coxeter groups are virtually
  special}, Adv. Math. \textbf{224} (2010), no.~5, 1890--1903. \MR{2646113}

\bibitem[HW16]{HW16}
Joseph Helfer and Daniel~T. Wise, \emph{Counting cycles in labeled graphs: the
  nonpositive immersion property for one-relator groups}, Int. Math. Res. Not.
  IMRN (2016), no.~9, 2813--2827. \MR{3519130}

\bibitem[JW16]{jankiewicz_16}
Kasia Jankiewicz and Daniel~T. Wise, \emph{Incoherent {C}oxeter groups}, Proc.
  Amer. Math. Soc. \textbf{144} (2016), no.~5, 1857--1866. \MR{3460148}

\bibitem[JZ17]{Ja17}
Andrei Jaikin-Zapirain, \emph{Approximation by subgroups of finite index and
  the {H}anna {N}eumann conjecture}, Duke Math. J. \textbf{166} (2017), no.~10,
  1955--1987. \MR{3679885}

\bibitem[JZ21]{Ja21}
\bysame, \emph{The universality of {H}ughes-free division rings}, Selecta Math.
  (N.S.) \textbf{27} (2021), no.~4, Paper No. 74, 33. \MR{4292784}

\bibitem[JZ24]{Ja24}
\bysame, \emph{Free {$\Bbb Q$}-groups are residually torsion-free nilpotent},
  Ann. Sci. \'Ec. Norm. Sup\'er. (4) \textbf{57} (2024), no.~4, 1101--1133.
  \MR{4773301}

\bibitem[JZLA20]{JL20}
Andrei Jaikin-Zapirain and Diego L\'{o}pez-\'{A}lvarez, \emph{The strong
  {A}tiyah and {L}\"{u}ck approximation conjectures for one-relator groups},
  Math. Ann. \textbf{376} (2020), no.~3-4, 1741--1793. \MR{4081128}

\bibitem[Kam19]{Ka19}
Holger Kammeyer, \emph{Introduction to {$\ell^2$}-invariants}, Lecture Notes in
  Mathematics, vol. 2247, Springer, Cham, 2019. \MR{3971279}

\bibitem[KKW22]{kielak_22}
Dawid Kielak, Robert Kropholler, and Gareth Wilkes, \emph{{$\ell^2$}-{B}etti
  numbers and coherence of random groups}, J. Lond. Math. Soc. (2) \textbf{106}
  (2022), no.~1, 425--445. \MR{4454494}

\bibitem[KL24]{KL23}
Dawid Kielak and Marco Linton, \emph{Virtually free-by-cyclic groups}, Geom.
  Funct. Anal. \textbf{34} (2024), no.~5, 1580--1608. \MR{4792841}

\bibitem[KLL09]{KLL09}
Peter Kropholler, Peter Linnell, and Wolfgang L\"{u}ck, \emph{Groups of small
  homological dimension and the {A}tiyah conjecture}, Geometric and
  cohomological methods in group theory, London Math. Soc. Lecture Note Ser.,
  vol. 358, Cambridge Univ. Press, Cambridge, 2009, pp.~272--277. \MR{2605183}

\bibitem[KMS60]{karrass_60}
A.~Karrass, W.~Magnus, and D.~Solitar, \emph{Elements of finite order in groups
  with a single defining relation}, Comm. Pure Appl. Math. \textbf{13} (1960),
  57--66. \MR{124384}

\bibitem[KS70]{KS70}
A.~Karrass and D.~Solitar, \emph{The subgroups of a free product of two groups
  with an amalgamated subgroup}, Trans. Amer. Math. Soc. \textbf{150} (1970),
  227--255. \MR{260879}

\bibitem[KS71]{KS71}
\bysame, \emph{Subgroups of {${\rm HNN}$} groups and groups with one defining
  relation}, Canadian J. Math. \textbf{23} (1971), 627--643. \MR{301102}

\bibitem[Lam78]{La77}
Kee~Yuen Lam, \emph{Group rings of {${\rm HNN}$} extensions and the coherence
  property}, J. Pure Appl. Algebra \textbf{11} (1977/78), no.~1-3, 9--13.
  \MR{466196}

\bibitem[Lew74]{Le74}
Jacques Lewin, \emph{Fields of fractions for group algebras of free groups},
  Trans. Amer. Math. Soc. \textbf{192} (1974), 339--346. \MR{338055}

\bibitem[Lin06]{Li06}
Peter~A. Linnell, \emph{Noncommutative localization in group rings},
  Non-commutative localization in algebra and topology, London Math. Soc.
  Lecture Note Ser., vol. 330, Cambridge Univ. Press, Cambridge, 2006,
  pp.~40--59. \MR{2222481}

\bibitem[Lin22]{linton_22}
Marco Linton, \emph{One-relator hierarchies}, 2022, arXiv:2202.11324.

\bibitem[LS01]{lyndon_01}
Roger~C. Lyndon and Paul~E. Schupp, \emph{Combinatorial group theory}, Classics
  in Mathematics, Springer-Verlag, Berlin, 2001, Reprint of the 1977 edition.
  \MR{1812024}

\bibitem[LS12]{LS12}
Peter~A. Linnell and Thomas Schick, \emph{The {A}tiyah conjecture and
  {A}rtinian rings}, Pure Appl. Math. Q. \textbf{8} (2012), no.~2, 313--327.
  \MR{2900171}

\bibitem[LW17]{LW17}
Larsen Louder and Henry Wilton, \emph{Stackings and the {$W$}-cycles
  conjecture}, Canad. Math. Bull. \textbf{60} (2017), no.~3, 604--612.
  \MR{3679733}

\bibitem[LW20]{louder_20}
\bysame, \emph{One-relator groups with torsion are coherent}, Math. Res. Lett.
  \textbf{27} (2020), no.~5, 1499--1511. \MR{4216595}

\bibitem[LW22]{LW22}
\bysame, \emph{Negative immersions for one-relator groups}, Duke Math. J.
  \textbf{171} (2022), no.~3, 547--594. \MR{4382976}

\bibitem[LW24]{LW21}
\bysame, \emph{Uniform negative immersions and the coherence of one-relator
  groups}, Invent. Math. \textbf{236} (2024), no.~2, 673--712. \MR{4728240}

\bibitem[Lyn50]{Lyn50}
Roger~C. Lyndon, \emph{Cohomology theory of groups with a single defining
  relation}, Ann. of Math. (2) \textbf{52} (1950), 650--665. \MR{47046}

\bibitem[Mag30]{magnus_30}
Wilhelm Magnus, \emph{\"{U}ber diskontinuierliche {G}ruppen mit einer
  definierenden {R}elation. ({D}er {F}reiheitssatz)}, J. Reine Angew. Math.
  \textbf{163} (1930), 141--165. \MR{1581238}

\bibitem[Mas06]{masters_06}
Joseph~D. Masters, \emph{Heegaard splittings and 1-relator groups}, unpublished
  paper, 2006.

\bibitem[Mil21]{millard_21}
Benjamin Millard, \emph{Stackings and one-relator products of groups}, Ph.D.
  thesis, University College London, 2021.

\bibitem[Min12]{Mi12}
Igor Mineyev, \emph{Submultiplicativity and the {H}anna {N}eumann conjecture},
  Ann. of Math. (2) \textbf{175} (2012), no.~1, 393--414. \MR{2874647}

\bibitem[New74]{Ca73}
M.~F. Newman (ed.), \emph{Proceedings of the {S}econd {I}nternational
  {C}onference on the {T}heory of {G}roups.}, Springer-Verlag, Berlin-New
  York,,, 1974, Held at the Australian National University, Canberra, August
  13--24, 1973., With an introduction by B. H. Neumann. \MR{344315}

\bibitem[Pio08]{Pi08}
Dmitri Piontkovski, \emph{Coherent algebras and noncommutative projective
  lines}, J. Algebra \textbf{319} (2008), no.~8, 3280--3290. \MR{2408318}

\bibitem[Rot09]{Rotbook}
Joseph~J. Rotman, \emph{An introduction to homological algebra}, second ed.,
  Universitext, Springer, New York, 2009. \MR{2455920}

\bibitem[Ser03]{serre_80}
Jean-Pierre Serre, \emph{Trees}, Springer Monographs in Mathematics,
  Springer-Verlag, Berlin, 2003, Translated from the French original by John
  Stillwell, Corrected 2nd printing of the 1980 English translation.
  \MR{1954121}

\bibitem[S{\v{S}}11]{SS10}
Mark Sapir and Iva {\v{S}}pakulov\'{a}, \emph{Almost all one-relator groups
  with at least three generators are residually finite}, J. Eur. Math. Soc.
  (JEMS) \textbf{13} (2011), no.~2, 331--343. \MR{2746769}

\bibitem[Swa69]{swan_69}
Richard~G. Swan, \emph{Groups of cohomological dimension one}, J. Algebra
  \textbf{12} (1969), 585--610. \MR{240177}

\bibitem[Swa04]{Sw04}
Gadde~A. Swarup, \emph{Delzant's variation on scott complexity}, 2004,
  arXiv:math/0401308.

\bibitem[Wei72]{We72}
C.~M. Weinbaum, \emph{On relators and diagrams for groups with one defining
  relation}, Illinois J. Math. \textbf{16} (1972), 308--322. \MR{297849}

\bibitem[Wis03]{wise_03}
Daniel~T. Wise, \emph{Nonpositive immersions, sectional curvature, and subgroup
  properties}, Electron. Res. Announc. Amer. Math. Soc. \textbf{9} (2003),
  1--9. \MR{1988866}

\bibitem[Wis05]{Wi05}
\bysame, \emph{The coherence of one-relator groups with torsion and the {H}anna
  {N}eumann conjecture}, Bull. London Math. Soc. \textbf{37} (2005), no.~5,
  697--705. \MR{2164831}

\bibitem[Wis20]{Wi20}
\bysame, \emph{An invitation to coherent groups}, What's next?---the
  mathematical legacy of {W}illiam {P}. {T}hurston, Ann. of Math. Stud., vol.
  205, Princeton Univ. Press, Princeton, NJ, 2020, pp.~326--414. \MR{4205645}

\bibitem[Wis22]{wise_20}
\bysame, \emph{Coherence, local indicability and nonpositive immersions}, J.
  Inst. Math. Jussieu \textbf{21} (2022), no.~2, 659--674. \MR{4386825}

\end{thebibliography}

\end{document}